\documentclass[11pt]{article}
\usepackage{flafter,amsmath,amssymb,latexsym,psfrag,graphicx,color,indentfirst,bm,amsthm}
\usepackage[T1]{fontenc}
\usepackage[utf8]{inputenc}
\usepackage[colorlinks,
linktocpage=true,
linkcolor=blue,
citecolor=green,
urlcolor=blue,
anchorcolor=black
]{hyperref}
\usepackage[margin=0.9in]{geometry}
\usepackage{float}
\usepackage[makeroom]{cancel}
\usepackage{mathtools}
\usepackage[numbers,sort&compress]{natbib}
\usepackage{appendix}
\usepackage{multirow}
\usepackage{caption}
\usepackage{subfigure}
\usepackage{graphicx}
\usepackage{dutchcal}
\usepackage{comment}
\allowdisplaybreaks
\usepackage{tikz}
\tikzset{mynode/.style={inner sep=2pt,fill,outer sep=0,circle}}
\usetikzlibrary{fadings}
\usepackage{amsmath,amssymb,braket,tikz}
\usetikzlibrary{tikzmark,calc}
\usetikzlibrary{arrows.meta,positioning,calc,decorations.pathreplacing,shadings}

\usepackage{microtype}
\newtheorem{theorem}{Theorem}[section]
\newtheorem{lemma}{Lemma}[section] 
\newtheorem{corollary}{Corollary}[section] 
\newtheorem{proposition}{Proposition}[section] 
\newtheorem{assumption}{Assumption} 
\newtheorem{remark}{Remark}[section]  

\numberwithin{equation}{section}

\usepackage{mathrsfs}
\newsavebox\foobox
\newlength{\foodim}
\newcommand{\slantbox}[2][0]{\mbox{%
        \sbox{\foobox}{#2}%
        \foodim=#1\wd\foobox
        \hskip \wd\foobox
        \hskip -0.5\foodim
        \pdfsave
        \pdfsetmatrix{1 0 #1 1}%
        \llap{\usebox{\foobox}}%
        \pdfrestore
        \hskip 0.5\foodim
}}
\def\Laplace{\slantbox[-.45]{$\mathscr{L}$}}

\numberwithin{equation}{section}
\numberwithin{figure}{section}

\newcommand\obullet[1]{\ThisStyle{\ensurestackMath{%
  \stackon[1pt]{\SavedStyle#1}{\SavedStyle\kern.6\LMpt\bullet}}}}
\newcommand\ocirc[1]{\ThisStyle{\ensurestackMath{%
  \stackon[1pt]{\SavedStyle#1}{\SavedStyle\kern.6\LMpt\circ}}}}
  
\makeatletter
\renewcommand{\@maketitle}{%
  \begin{center}
    {\LARGE \bfseries \@title \par}
    \vskip 1em
    {\large \@author}
  \end{center}
}
\makeatother

\title{Functional--Analytic Justification of the Time-Domain Foldy--Lax Approximation for Dispersive Acoustic Media: \\ A Feynman-Diagram Viewpoint}
\author{Arpan Mukherjee\footnote{MSU-BIT SMBU Joint Research Center of Applied Mathematics, Shenzhen MSU-BIT University, Shenzhen, People's Republic of China (arpan.mukherjee@smbu.edu.cn and arpanmath99@alumni.iitm.ac.in).} \ and Mourad Sini\footnote{Radon Institute (RICAM), Austrian Academy of Sciences, Altenbergerstrasse 69, A-4040, Linz, Austria (mourad.sini@oeaw.ac.at). This author is partially supported by the Austrian Science Fund (FWF): P36942.}}

\begin{document}
\maketitle
\begin{abstract}

\bigskip

This work provides a rigorous functional-analytic justification for a time-domain Foldy-Lax framework that describes multiple acoustic scattering by a cluster of dispersive resonators (modeling gas-filled bubbles), explicitly incorporating dispersion via the Minnaert resonance. The model is formulated as a delayed-coupled hyperbolic system for bubble amplitude interactions.
\newline
We combine time-domain integral equations, Laplace transforms, and Hardy-Sobolev space techniques to analyze this system, establishing its unique solvability in anisotropic Hilbert spaces, with solutions expressed as convergent Neumann series of convolution operators. We derive geometric decay of truncation errors for resonant incident waves and quantify the contribution of $N$-th order multi-scattering, showing it scales with \(\varepsilon^{N(1-p)+1}\) (relating bubble radius \(\varepsilon\) and inter-bubble distance scaling as $\varepsilon^p$, $p<1$). This dominates the measurement errors, which are of order $\varepsilon^2$, thereby allowing us to capture fields generated by inter-bubble interactions of order $N<\frac{1}{1-p}$. This provides a quantitative relation between the spectra band width of the source field, the closeness distance between the bubbles and the order $N$ of the relevant interactions between the bubbles.
\newline
Furthermore, a novel connection to Feynman diagrams maps multi-scattering paths to diagrammatic vertices and propagators, simplifying the interpretation of higher-order interactions and kinematic constraints. This framework advances accurate transient wave prediction in dispersive media, with implications for cavitation therapy, seismic imaging, and metamaterial engineering.

\bigskip
    
\end{abstract}

\section{Introduction and statement of the results}
\subsection{Introduction}

Multiple scattering phenomena are ubiquitous across classical physics, as in acoustics, electromagnetics, and elasticity. In acoustics, scattering by arrays of particles (e.g., bubbles, droplets) governs wave propagation in oceanic media, medical ultrasound imaging, and acoustic cavitation therapy—a transformative technique where focused ultrasound induces bubble oscillation to ablate tumors or enhance drug delivery to deep tissues. Electromagnetics relies on multiple scattering for antenna design, metamaterial engineering, medical imaging and radar cross-section analysis, while elasticity use it to characterize wave propagation in composite materials and geological formations (e.g., seismic imaging of oil reservoirs). Historically, these phenomena have been studied in two primary frameworks: time-harmonic (frequency-domain) and time-domain settings. Time-harmonic models, exemplified by the Helmholtz equation, simplify analysis via Fourier transformation but fail to capture transient effects like pulse propagation and wave dispersion—critical for real-world applications such as non-destructive testing (where transient pulses detect microcracks) and cavitation therapy (where bubble response to short ultrasound bursts dictates therapeutic efficacy). Time-domain models, by contrast, preserve temporal dynamics but introduce mathematical complexity due to delayed interactions between scatterers and dispersive material responses.
A cornerstone of multiple scattering theory is the Foldy-Lax approximation, initially developed for isotropic scatterers in quantum mechanics and later adapted to classical physics by Foldy and Lax, see \cite{Foldy} and \cite{Lax} respectively. This approximation reduces the collective scattering problem to a system of coupled equations for individual scatterer amplitudes, accounting for pairwise interactions while retaining computational tractability, see \cite{Martin-1, Martin-2}, for classical multiple scattering foundations.  
\bigskip

There is a large literature on such approximations in the time-harmonic regime. In the recent years, there was a revived interest in deriving such approximation in the time-domain regimes, see \cite{Kachanovska, S-W-Y, S-W}. However, these classical Foldy-Lax models often neglect (or are free from) dispersion—an inherent property of scatterers like gas-filled bubbles, where bulk modulus and mass density scaling induce frequency-dependent wave propagation. In bubbly media, for instance, linearized bubble dynamics (governed by Minnaert resonance) lead to dispersive acoustic wavefields, rendering standard time-harmonic Foldy-Lax models inadequate for transient analysis and limiting their applicability to narrow frequency bands. This gap motivates the need for a time-domain Foldy-Lax framework that rigorously incorporates dispersion, ensuring accurate prediction of transient wave behavior in dispersive media.
\bigskip

Our recent work \cite{Arpan-Sini-JEEQ} addressed this limitation by analyzing a linearized, time-dependent acoustic model for gas-filled bubble clusters and estimating the nearby time-dependent wave field. In addition, in \cite{ Arpan-Sini-SIAP, Arpan-Sini-arxiv}, we derived time-domain effective models for bubbly media, i.e. for clusters of bubbles distributed in 3D domains or 2D surfaces, showing that resonance-induced dispersion manifests as convolution terms in amplitude equations—consistent with our Foldy-Lax system’s delayed couplings. 
\bigskip

The approximating fields provided by the Foldy-Lax approximation incorporate a cascade of fields generated by mutual interactions between the bubbles of different orders. A natural question arises regarding the accuracy of this Foldy-Lax approximation. Namely, under which condition, can we guarantee that the field, measured away from the cluster of the bubbles, is generated after a given $N$ interactions between few of these bubbles? Related to this question, in \cite{a-a-m}, we confirmed that Foldy-Lax approximation is most reliable near scattering resonances (e.g., Minnaert resonance in bubbles). Precisely, we showed, with quantitative estimates, that closer is the used incident frequency to the Minnaert resonant one, larger is the order $N$ of interactions between the bubbles that we can capture from near or far measurement zones. 
\bigskip

This observation validates our focus on dispersive regimes where resonance-driven dispersion dominates.  In the current work, we combine time-domain integral equations, Laplace transform techniques, and functional analysis
(e.g., Hardy spaces, convolution operators), to analyze the delayed-coupled hyperbolic system for bubble amplitudes, modeling the Foldy-Lax system, explicitly accounting for dispersion and inter-bubble time delays. This approach extends the analysis of the accuracy of Foldy-Lax approximation to dispersive settings, enabling rigorous analysis of transient wave scattering—such as the sequence of bubble oscillations. Our core results validate the time-domain Foldy-Lax approximation for dispersive bubbly media and quantify its performance via three key findings: First, we prove that the delayed-coupled hyperbolic system for bubble amplitudes admits a unique solution in anisotropic Hilbert spaces ($H^r_{0,\sigma}$), with the solution expressed as a convergent Neumann series of convolution powers. Second, we derive explicit bounds for the truncation error of the Foldy-Lax series, showing that the remainder term decays geometrically with the number of scattering events $(N)$, i.e. $N$ interactions, when the incident wave’s spectral band width lies within the bubble’s resonance band. Third, we demonstrate that the difference between $N$-th and ($N-1$)-th order scattered fields—corresponding to $N$-fold multi-scattering—scales with $\varepsilon^{N(1-p)+1}$ (where $\varepsilon$ is bubble's radius and $p$ is inter-bubble distance scaling, i.e. the mutual distance between the bubbles is of the order $\varepsilon^p$, with $p<1$), while the error caused in estimating the scattered field, in the far or near zones, is of the order $\varepsilon^2$. As a conclusion, these measured fields are indeed generated  after the $N$-th interaction of the bubbles where $N<\frac{1}{1-p}$ as then $N(1-p)+1<2$. This confirms that higher-order paths contribute meaningfully to transient responses near resonance ( i.e. as soon as the incident field has a non zero spectral band including the Minneart resonance and that the bubbles are close to each other). 
\bigskip

Another novel contribution of our work is the connection between the Foldy-Lax approximation and Feynman diagrams—a powerful tool from quantum field theory (QFT) used to visualize and compute particle interactions. In QFT, Feynman diagrams represent scattering processes as paths of particle propagation (propagators) connected by interaction vertices, with each diagram corresponding to a term in a perturbation series. We map the Foldy-Lax amplitude equations to Feynman diagrams by: (1) interpreting inter-bubble scattering events (where a wave from one bubble induces oscillation in another) as vertices; (2) modeling wave propagation between bubbles as retarded propagators (analogous to Helmholtz-type propagators in QFT); and (3) identifying convolution powers of the scattering kernel as multi-scattering paths (as an example: bubble A → bubble B → bubble C corresponds to a two-vertex diagram). Each diagram path enforces kinematic constraints (e.g., time delay conservation) via Dirac delta functions, corresponding to the role of delta functions in QFT Feynman rules, \cite{Nair, Nelison, Wuethrich}. 
\bigskip

Our contributions can be summarized in three aspects: (1) Functional analytic estimation : We utilize Hardy-Sobolev spaces of analytic operator-valued functions and Laplace transform isometries to prove boundedness of the scattering kernel, ensuring the convergence of the Foldy-Lax perturbation series. This addresses a critical gap in classical models, which often assume series convergence without functional justification. (2) Dispersion-Inclusive Formulation: By incorporating the Minnaert resonance condition (governing bubble oscillation) into the hyperbolic amplitude system, we explicitly model dispersion via frequency-dependent scattering coefficients. (3) Diagrammatic Perturbation Tracking: The Feynman diagram correspondence enables systematic tracking of multi-scattering paths, with each diagram term corresponding to a convolution power of the scattering kernel. This not only simplifies the interpretation of higher-order scattering effects but also provides a quantitative framework to estimate truncation errors, ensuring the Foldy-Lax approximation’s accuracy for transient problems.
\bigskip

\subsection{Problem Setting and the Mathematical Model}

\noindent
We consider a collection $\mathrm{D} = \bigcup\limits_{i=1}^M D_i$ of $M$ connected, bounded $\mathcal{C}^2$-regular inclusions in $\mathbb R^3$, which model acoustic bubbles. The total acoustic wave field $u$, defined as \( u := u^{\textit{sc}} + u^{\textit{in}} \), is governed by the following transmission problem:
    \begin{align}\label{hyperbolic}
        \begin{cases}   
            \kappa^{-1}(x)u_{tt}- \operatorname{div} \rho^{-1}(x)\nabla u = \textcolor{black}{\lambda(t)\delta_{x_0}(x)} & \textit{in}\;  (\mathbb{R}^{3}\setminus \partial\mathrm{D}) \times (0,T),\\
            u\big|_{+} = u\big|_{-}  & \text{on}\; \partial D_i,\\
            \rho_c^{-1}\partial_\nu u\big|_{+} = \rho_{b,i}^{-1} \partial_\nu u\big|_{-} & \text{on} \; \partial D_i, \\
            u(x,0)= u_t(x,0) = 0 & \text{for}\; x \in \mathbb{R}^3,
        \end{cases}
    \end{align}
where $u^\textit{in}$ is the incident wave generated by a point source at $x_0 \notin \overline{D}$, given by
\begin{align}
    u^{\textit{in}}(x, t, x_0) := \rho_c\frac{\lambda\left(t - c_0^{-1} \vert x - x_0 \vert\right)}{4\pi\Vert x - x_0 \Vert}\; \text{with}\ c_0=\sqrt{\frac{\kappa_c}{\rho_c}}
\end{align}
while $u^\textit{sc}:= u-u^{\textit{in}}$ denotes the scattered field. The source term involves a causal signal $\lambda(t)$ and the Dirac delta function $\delta_{x_0}(x)$, centered at \(x_0 \in \mathbb R^3\setminus \overline{D}.\)
\newline
The mass density $\rho$ and bulk modulus $\kappa$ are defined piecewise as follows: 
\begin{align}
    \rho = \rho_b\bm{1}_{\mathrm{D}} + \rho_{c}\bm{1}_{\mathbb{R}^{3}\setminus\overline{\mathrm{D}}},\; \kappa = \kappa_b\bm{1}_{\mathrm{D}} + \kappa_c\bm{1}_{\mathbb{R}^{3}\setminus\overline{\mathrm{D}}}, 
\end{align}
where $\rho_b,\kappa_b$ and $\rho_c,\kappa_c$ are the positive constants representing the mass density and bulk modulus inside the bubbles and the background medium, respectively. Here, $\bm{1}_A$ is the characteristic function of a set $A$ \big($\bm{1}_A(x) =1$ if $x\in A$ and $\bm{1}_A(x)=0$ otherwise\big). The Dirichlet and Neumann traces on $\partial D$ are defined as:
$$u \big|_{\pm}(x,t) := \lim_{h\to 0}u(x\pm h\nu_x,t),\; \partial_\nu u \big|_{\pm}(x,t) := \lim_{h\to 0}\nabla u(x\pm h \nu_x,t)\cdot \nu_x,$$
where $\nu_x$ is the outward normal vector to $\partial\mathrm{D}$ at $x$. The subscripts $+$ and $-$ denote the exterior and interior traces, respectively.
\newline
Each bubble \(D_i\) is described by \(D_i = \varepsilon B_i + z_i\), where \(i = 1, 2, \ldots, M\), \(z_i \in \mathbb R^3\) is the center and \(B_i\) is a uniformly bounded, \(\mathcal{C}^2\) regular domain containing the origin and has a volume on the order of unity, $\operatorname{vol}(B_i) \sim 1$. 
\newline
We define \(d\) as the minimum distance between bubbles, denoted by \(d_{ij} = \operatorname{dist}(D_i, D_j)\) for \(i \neq j\), where \(\operatorname{dist}\) represents the distance function, i.e.,
    \begin{align}
         d := \min_{i,j} d_{ij}.
    \end{align}
Furthermore, we denote \(\varepsilon\) as the maximum diameter among the microbubbles, given by:
    \begin{align}
         \varepsilon := \max_{i} \operatorname{diam}(D_i),
    \end{align}
where \(\operatorname{diam}\) represents the diameter function.
\begin{assumption}\label{as111}
We model the bubbles as gas-filled inclusions. This means, their bulk modulus $\kappa_{b,i}$ and mass density $\rho_{b,i}$ are significantly smaller than those of the surrounding medium. To model this, we design the bubble radius such that the material parameters scale as follows:
    \begin{equation}\label{scale}
        \big[\kappa(x), \rho(x)\big] := 
        \begin{cases}
            [\kappa_c, \rho_{\mathrm{c}}] & \textit{in} \; \mathbb{R}^3 \setminus D_i, \\
            [\kappa_{b,i}, \rho_{b,i}] = \big[\overline{\kappa}_{b,i} \varepsilon^2, \overline{\rho}_{b,i} \varepsilon^2\big] & \textit{in} \; D_i,
        \end{cases}
    \end{equation}
where \(\overline{\kappa}_{b,i}\) and \(\overline{\rho}_{b,i}\) are constant functions independent of \(\varepsilon\). The parameters of the homogeneous background medium, $\kappa_c$ and $\rho_c$, are also constant and independent of $\varepsilon.$ With such scaling, we observe that the speed of propagation in \(\bigcup\limits_{i=1}^MD_i\) and \(\mathbb{R}^{3} \setminus \overline{\bigcup\limits_{i=1}^MD_i}\) is of order 1, i.e., \(\frac{\kappa}{\rho} \sim 1\).    
\end{assumption}
\begin{assumption}\label{as11}
In this paper, we focus on the following regimes for modeling bubbles distributed throughout a three-dimensional bounded $\mathcal{C}^2$-regular domain $\bm{\Omega}$ of unit volume:
\begin{align} \label{regimes}
M \le M_{\max} \quad \text{and} \quad d = \widetilde{d}\ \varepsilon^p,\; 0\le p<1, \quad \varepsilon \ll 1,
\end{align}
where $M_{\max}$ and $\widetilde{d}$ are positive constants independent of $\varepsilon$ with $M_{\operatorname{max}}$ to be specified later.
\end{assumption}

\subsection{Function Spaces and Mathematical Preliminaries}

\noindent
We begin by defining the necessary function spaces needed for our forthcoming analysis following the framework and notations from \cite{A-P, X-A-M,Arpan-Sini-arxiv}. Let $X$ be a separable Hilbert space. We then define \( \mathcal{D}'(\mathbb{R}; X) \) as the space of \( X \)-valued distributions on the real line. Moreover, we define the following space
\[
    \mathcal{L}'(\mathbb{R}, X) := \left\{ f \in \mathcal{D}'(\mathbb{R}, X) : e^{-\sigma t} f \in \mathcal{S}'(\mathbb{R}_+, X) \ \text{for some} \ \sigma > 0 \right\},
\]
where \( \mathcal{S}'(\mathbb{R}_+; X) \) is the space of \( X \)-valued tempered distributions in $\mathbb R$ having support in $[0,\infty).$ Therefore, for a function \( f \in \mathcal{L}'(\mathbb{R}, X) \cap L^1(\mathbb{R}, X) \) and any $\sigma>\sigma_f$ the Laplace transform is given by
\[ 
    \widehat{f}(s) 
    = \Laplace[f](s) 
    := \int_0^{+\infty} f(t) e^{-s t}  dt, 
\]
is well defined and holomorphic on $\Re s > \sigma_f$ with $\sigma_f$ is the abscissa of convergence.  In particular, we fix $\sigma\ge\sigma_0$ with $\sigma_0=\max\{0,\sigma_f\}$, and parametrize $s=\sigma+i\omega$ to evaluate $\widehat f(\sigma+i\omega)$ as the boundary values on the vertical line $\Re s=\sigma$.
\newline
Let \(\bm{\omega}(t)\) be a positive measurable function in $\mathbb{R}^+$. We define the function space
\begin{align}
    L^2(0,\infty,\bm{\omega}(t)dt,X) := \left\{f:[0,\infty)\to X :\int_{0}^\infty \Vert f(t)\Vert^2_X \bm{\omega}(t) dt<\infty\right\},
\end{align}
where \(X\) is a separable Hilbert space.
\newline
Let us now consider \(\bm{\nu}\) to be a positive Borel measure on $[0,\infty)$ satisfying the doubling condition \(R_{\bm{\nu}}:=\sup\limits_{t>0}\frac{\bm{\nu}[0,2t)}{\bm{\nu}[0,t)}<\infty\), and has no atom on $0$ i.e. \(\bm{\nu}(\{0\})=0.\)  We then define the Zen space on the right-half plane as
\begin{align}
    A_{\bm{\nu}}^2(X):= \left\{ F: \mathbb{C}^+\to X\ \text{analytic}: \ \Vert F\Vert^2_{A_{\bm{\nu}}^2(X)} := \int_{\overline{\mathbb{C}^+}} \Vert F(\xi+i\omega)\Vert^2_X d\bm{\nu}(\xi)d\omega<\infty \right\},
\end{align}
where \(s=\xi+i\omega\) with \(\Re s = \xi\ge0\) and \(\omega\in \mathbb{R}.\) We then state the following lemma.
\begin{lemma}\label{l1}
    \cite{A-P} Suppose that \(\bm \omega\) is given as a weighted Laplace transform
    \begin{align}
        \bm \omega(t) = 2\pi \int_0^\infty e^{-2\xi t}d\bm{\nu}(\xi).
    \end{align}
    Then the Laplace transform provides an isometric map \(L^2(0,\infty,\bm{\omega}(t)dt,X) \to A_{\bm{\nu}}^2(X).\)
\end{lemma}
\noindent
To proceed, we fix $\sigma>0$ and write $s=\sigma+i\omega$. Accordingly, we define the shifted half-plane \(\mathbb{C}^+_{\sigma}:= \{s\in \mathbb{C}: \Re s>\sigma\}\) and the shifted Zen space as follows:
\begin{align}
    A_{\bm{\nu},\sigma}^2(X):= \left\{ F: \mathbb{C}_{\sigma}^+\to X\ \text{analytic}: F_\sigma(s)=F(s+\sigma) \in A_{\bm{\nu}}^2(X) \right\}.
\end{align}
The following corollary is straightforward to prove, so we state it without proof.
\begin{corollary}
    \(A_{\bm{\nu},\sigma}^2(X)\) is a Hilbert space isometrically isomorphic to \(A_{\bm{\nu}}^2(X)\) via the unitary operator \(U: A_{\bm{\nu},\sigma}^2(X)\to A_{\bm{\nu}}^2(X)\) defined by \((UF)(s):=F(s+\sigma)\) with the norm
\begin{align}
    \Vert F\Vert^2_{A_{\bm{\nu},\sigma}^2(X)} = \Vert F_\sigma\Vert^2_{A_{\bm{\nu}}^2(X)}.
\end{align}
\end{corollary}
\noindent
We recall the Hardy space:
\begin{align}
 \nonumber & H^p(\mathbb{C}^+,\mathcal{L}(X)) 
 \\ &:= \left\{ F: \mathbb{C}^+\to \mathcal{L}(X)\ \text{is analytic}:  \Vert F\Vert_{H^p(\mathbb{C}^+,\mathcal{L}(X))}
:= \Big(\sup_{\xi>0}\int_{-\infty}^\infty \Vert F(\xi+i\omega)\Vert^p_{\operatorname{op}}d\omega\Big)^{\frac{1}{p}}<\infty \right\},
\end{align}
and in a similar manner, we consider $H^\infty(\mathbb{C}^+,\mathcal{L}(X))$ as set of all bounded and analytic function in $\mathbb{C}^+$ with the norm $\Vert F\Vert_\infty:= \sup\limits_{s\in \mathbb{C}^+}\Vert F(s)\Vert_{\operatorname{op}}.$ Accordingly, we define the Hardy space $H^\infty(\mathbb{C}_{\sigma}^+,\mathcal{L}(X))$ on the shifted half-plane $\mathbb{C}_{\sigma}^+.$ Next, we state the following lemma.
\begin{lemma}\label{l2}
    \cite{A-P} Let \(G\) belong to the Hardy space \(H^\infty(\mathbb{C}^+,\mathcal{L}(X)).\) Then, for \(s\in \mathbb{C}^+\) and \(F\in A_{\bm{\nu}}^2(X),\) the multiplication operator \(M_G\) defined by
    \begin{align*}
        (M_G F)(s) = G(s) F(s)
    \end{align*}
    is bounded on \(A_{\bm{\nu}}^2(X)\) with \(\Vert M_G\Vert\le \Vert G\Vert_\infty.\)
\end{lemma}
\noindent
Moreover, if \(F\in H^\infty(\mathbb{C}_{\sigma}^+,\mathcal{L}(X))\), then it is straightforward to see that the multiplication operator \(M_G\) on \(A_{\bm{\nu},\sigma}^2(X)\) is unitarily equivalent to the multiplication operator \(M_{\widetilde{G}}\) on \(A_{\bm{\nu}}^2(X)\) with \(\widetilde{G}(s):= G(s+\sigma).\) Consequently, we have the following corollary.
\begin{corollary}
    The statements of Lemma \ref{l1} and Lemma \ref{l2} hold for the shifted half-plane and thus for the shifted Zen space \(A_{\bm{\nu},\sigma}^2(X).\)
\end{corollary}
\noindent
 For $\sigma >0$ and \( r \in \mathbb{Z}_{\ge 0} \), we define the following anisotropic Hilbert space as 
\[
    H_{0,\sigma}^r(0,\infty; X) := \left\{ f: f \in \mathcal{S}'(\mathbb{R}_+, X)\ \text{and}\ \Vert f \Vert^2_{H_{0,\sigma}^r(0,\infty; X)}: \sum\limits_{k=0}^r\int_{0}^{+\infty}  \Vert e^{-\sigma t}\partial_t^k f(\cdot,t) \Vert_X^2  dt < \infty \right\}.
\]
For $r\in \mathbb{R}$, we define the following Hardy–Sobolev space of analytic $X$-valued functions on $\mathbb{C}_\sigma^+$ as:
$$H_{r,\sigma}( X):= \left\{f: \mathbb{C}_\sigma^+\to X\ \text{is analytic}: \Vert f \Vert^2_{H_{r,\sigma}(X)} = \int_{-\infty}^{+\infty} (\sigma^2+\omega^2)^r \Vert f(\sigma+i\omega) \Vert_X^2  d\omega<\infty\right\}.$$
Then $H^\infty(\mathbb{C}_{\sigma}^+,\mathcal{L}(X))$ acts by multiplication on $H_{r,\sigma}$ and the Laplace transform establishes the isometric isomorphism (see \cite[Chapter 3]{sayas}, \cite[Section 3.2]{PM})
$$\Laplace:  H_{0,\sigma}^r(0,\infty; X) \overset{\text{unitary}}{\longrightarrow} H_{r,\sigma}( X).$$
Motivated by Lemma~\ref{l1} and Lemma~\ref{l2}, we state the following corollary.
\begin{corollary}\label{cor1.3}
    Let $G\in H^\infty(\mathbb{C}_{\sigma}^+,\mathcal{L}(X))$ and assume that $G$ extends continuously to the closed half-plane $\overline{\mathbb{C}_\sigma^+}.$ We define the multiplication operator $M_{G}: H_{r,\sigma}\to H_{r,\sigma}$ by
    $$\big(M_{G} F\big)(s):=G(s) F(s).$$
    Then $M_{G}$ is bounded on $H_{r,\sigma}(X)$ and $\Vert M_{G}\Vert_{\mathcal{L}\big(H_{r,\sigma};H_{r,\sigma}\big)}\le \sup\limits_{\Re s=\sigma}\Vert G(s)\Vert_{\operatorname{op}},$ where $\Vert \cdot \Vert_{\operatorname{op}}$ denotes the operator norm on $\mathcal{L}(X).$
\end{corollary}
\begin{proof}
  We begin by observing that the multiplication operator, $M_{G},$ is defined point-wise for all $s=\sigma+i\omega.$ This means that for every $ F\in H_{r,\sigma}$ and every $s=\sigma+i\omega$ point-wise, we have
  \begin{align}
      \Vert M_{G} F(s) \Vert_{X} \le \Vert G(s)\Vert_{\operatorname{op}}\Vert F(s) \Vert_{X} \le \Big(\sup\limits_{\omega\in \mathbb{R}}\Vert G(\sigma+i\omega)\Vert_{\operatorname{op}}\Big) \Vert F \Vert_{X}.
  \end{align}
  Since $G\in H^\infty(\mathbb{C}_{\sigma}^+,\mathcal{L}(X))$ and due to the fact that it admits a continuous extension to the boundary line $\Re s = \sigma,$ we use Phragmén–Lindelöf principle to conclude that
  $$\sup\limits_{\omega\in \mathbb{R}}\Vert G(\sigma+i\omega)\Vert_{\operatorname{op}} = \sup\limits_{\Re s>\sigma}\Vert G(s)\Vert_{\operatorname{op}} = \sup\limits_{\Re s=\sigma}\Vert G(s)\Vert_{\operatorname{op}}.$$
  \newline
  Therefore, from the definition of norm on $H_{r,\sigma},$ we see that
  \begin{align}
    \Vert M_{G} F \Vert_{H_{r,\sigma}}^2 = \int_{-\infty}^\infty |s|^{2r}\big\|G(s) F(s)\big\|^2dw \le \left(\sup_{\Re s=\sigma} \big\|G(s)\big\|_{\operatorname{op}}\right)^2 \Vert  F \Vert_{H_{r,\sigma}}^2.
\end{align}
By taking square roots on both sides, $M_{G}$ is bounded on $H_{r,\sigma}$ and $\Vert M_{G}\Vert_{\mathcal{L}\big(H_{r,\sigma};H_{r,\sigma}\big)}\le \sup\limits_{\Re s=\sigma}\Vert G(s)\Vert_{\operatorname{op}}.$ 
\end{proof}
\noindent
We define the Banach space \( \mathcal{B}([0,T]) \oplus L^1([0,T]) \), where \( \mathcal{B}([0,T]) \) is the space of finite signed Borel measures on \([0,T]\), equipped with the total variation norm \( \Vert\mu\Vert_{TV} = \sup_i \sum |\mu(E_i)| \) taken over all finite partitions \( (E_i) \) of \([0,T]\), and \( L^1([0,T]) \) is the space of Lebesgue integrable functions with the norm \(\displaystyle \Vert f\Vert_1 = \int_0^T |f(t)|\,dt \).
\newline
We then recall the following Helmholtz decomposition (see \cite{D-L, raveski} for details):
\begin{align} \label{decomposition-introduction}
\big(L^2(D)\big)^3 = \mathbb{H}_{0}(\operatorname{div} 0, D) \oplus \mathbb{H}_{0}(\operatorname{curl} 0, D) \oplus \nabla \mathbb{H}_{\mathrm{arm}},
\end{align}
where
\begin{align}
    \begin{cases}
        \mathbb{H}_{0}(\operatorname{div} 0, D) = \big\{ u \in \mathbb{H}(\operatorname{div}, D): \operatorname{div} u = 0 \text{ in } D, \; u \cdot \nu = 0 \text{ on } \partial D \big\}, \\
        \mathbb{H}_{0}(\operatorname{curl} 0, D) = \big\{ u \in \mathbb{H}(\operatorname{curl}, D): \operatorname{curl} u = 0 \text{ in } D, \; u \times \nu = 0 \text{ on } \partial D \big\}, \\
        \nabla \mathbb{H}_{\mathrm{arm}} = \big\{ u \in \big(L^2(D)\big)^3: \exists \varphi \in H^1(D) \text{ with } \Delta \varphi = 0 \text{ and } u = \nabla \varphi \big\}.
    \end{cases}
\end{align}
The magnetization operator is defined by
\[
\mathbb{M}^{(0)}_{B_j}[f](x) := \nabla \int_{B_j} \nabla_y \frac{1}{4\pi |x - y|} \cdot f(y)  dy.
\]
It is well known (see, for instance, \cite{friedmanI}) that \(\mathbb{M}^{(0)}_{B_j}: \nabla \mathbb{H}_{\operatorname{arm}} \rightarrow \nabla \mathbb{H}_{\operatorname{arm}}\) admits a complete orthonormal basis of eigenfunctions, denoted by \(\big(\lambda^{(3)}_{n_j}, e^{(3)}_{n_j}\big)_{n \in \mathbb{N}}\). We define \(\lambda_1^{(3)} := \min\limits_j\max\limits_n \lambda_{n_j}^{(3)}\).

    \subsection{Statement of the Result}\label{Main-results} 

\begin{theorem}\label{maintheorem}   
    Let \(\bm{Y}(t) = (Y_i)_{i=1}^M\) satisfy the following hyperbolic system:
    \begin{align} \label{matrix2}
        \begin{cases}
              \mathcal{A} \frac{d^2}{dt^2}\bm{Y}(t) + \bm{Y}(t) = \mathbcal{F}(t) \quad \text{in } (0, T), \\
              \bm{Y}(0) = \frac{d}{dt} \bm{Y}(0) = 0,
        \end{cases}
    \end{align}
where the operator \(\mathcal{A}: (L^2_r)^M \to (L^2_r)^M\) is defined by
\begin{align}
    \mathcal{A} = \mathcal{A}(t):=
     \begin{pmatrix}
        \omega_M^2 & \dots & q_{1M}\mathcal{T}_{-\tau_{1M}} \\
        \vdots & \ddots & \vdots \\
        q_{M1}\mathcal{T}_{-\tau_{M1}} & \dots & \omega_M^2
    \end{pmatrix},
\end{align}
with \(L^2_r := \{ f \in L^2(-r, T) : f = 0 \text{ in } (-r, 0) \}\), translation operators \(\mathcal{T}_{-\tau_{ij}}\) given by \(\mathcal{T}_{-\tau_{ij}}(f)(t) := f(t - \tau_{ij})\), and \(r := \max\limits_{i \neq j} \tau_{ij}\) where \(\tau_{ij} := c_0^{-1} |z_i - z_j|\). Furthermore, for \(1 \le i \neq j \le M\), we have
\begin{equation}\label{b_j}
q_{ij} := \frac{C_j}{4\pi |z_i - z_j|}, \quad \text{with} \quad C_j := \frac{\rho_c}{\overline{\kappa}_{b,j}} \operatorname{vol}(B_j) \varepsilon + \mathcal{O}(\varepsilon^3).
\end{equation}
Neglecting the error terms, we define the scattering coefficient \(C_j = C^{(0)}_j \varepsilon\), where \(C^{(0)}_j := \frac{\rho_c}{\overline{\kappa}_{b,j}} \operatorname{vol}(B_j)\), and having identical shapes and material properties of the bubbles, let  \(C^{(0)} := C^{(0)}_j\). The constant \( \omega_M^2 = \frac{\rho_c}{2 \overline{\kappa}_{\mathrm{b},j}} \Lambda_{\partial B_j} \) is the inverse square of the Minnaert frequency, where
\(\displaystyle\Lambda_{\partial B_j} := \frac{1}{|\partial B_j|} \int_{\partial B_j} \int_{\partial B_j} \frac{(x - y) \cdot \nu_{x}}{|x - y|}  d\sigma_{x}  d\sigma_{y}\) is a geometric constant. In addition, we denote $\mathbcal{F}:=\big(\mathcal{F}_1,\mathcal{F}_2,\ldots,\mathcal{F}_M\big)^{\textit{Tr}}$ with $\mathcal{F}_i:=\partial_t^2 u^{\textit{in}}(z_i,t)$ for $i=1,2,\ldots,M,$ where each $\mathcal{F_i}$ is causal, supported in $[0,T],$ and belongs to $H^r_{0,\sigma}\big(0,T;L^2_{\operatorname{loc}}(\mathbb{R}^3)\big)$ with $r\ge 10$ and $\sigma>\sigma_0.$

\smallskip

\noindent
\begin{enumerate}
    \item For \( \Re s\ge\sigma_0 > \omega_M^{-1} \), under the conditions
\begin{align}
    \varepsilon \cdot \max_{1 \leq i \leq M} \sum_{j \neq i} \frac{C^{(0)}_j}{4\pi |z_i - z_j|} &< \omega_M^2 \left(1 - \frac{1}{\omega_M^2 \sigma_0^2}\right),
\end{align}
the solution \(\bm{Y}(t)\) has the time-domain representation
\begin{align}
    \bm{Y}(t) = \sum_{n=0}^\infty (-1)^n K^{*n} * V(t).
\end{align}
Here, the convolution powers \(K^{*n}\) are well-defined as elements of the Banach algebra 
\newline
\(\left( \mathcal{B}([0,T]) \oplus L^1([0,T]) \right)^{M \times M}\), and their Laplace transform satisfy \(\mathcal{L}[K^{*n}] = T(s)^n\), where each scalar entry of the matrix \(T(s)\) is given by \(T_{ij}(s) = \frac{s^2}{\omega_M^2 s^2 + 1} q_{ij} e^{-s \tau_{ij}}\), which is analytic for \(\Re s \ge \sigma_0\). The vector \(V(t)\) has scalar entries \(V_i(t)\), for \(i = 1, 2, \dots, M\), defined by
\(\displaystyle V_i(t) = \frac{1}{\omega_M^2} \mathcal{F}_i(t) - \frac{1}{\omega_M^3} \int_0^t \sin\left( \frac{t - \tau}{\omega_M} \right) \mathcal{F}_i(\tau)  d\tau.\)
\newline
For any fixed $N\in \mathbb{N},$ we define the remainder term
\begin{align}
    R_N(t) := \sum_{n=N+1}^\infty (-1)^n K^{*n} * V(t).
\end{align}
Then, under the Assumptions \ref{as111} and \ref{as11}, we have the estimate
\begin{align}\label{d}
    \|R_N\|_{H^r_{0,\sigma}} \leq \frac{\alpha_\infty^{N+1}}{1-\alpha_\infty \varepsilon^{1-p}}\varepsilon^{(N+1)(1-p)}\frac{\sigma_0^2}{\omega_M^2\sigma_0^2-1}\|\mathbcal{F}\|_{H^r_{0,\sigma}},
\end{align}
\begin{align}
    \text{where}\; \alpha_\infty\leq \frac{\sigma_0^2}{\widetilde{d}\big(\omega_M^2\sigma_0^2-1\big)}(M-1)C^{(0)}.
\end{align}
\item Moreover, the corresponding scattered field, $u^{\textit{sc}}:= u-u^{\textit{in}},$ satisfy
\begin{align}\label{e}
    \|u^{\textit{sc}}(x,t)-u^{\textit{sc},N}(x,t)\|_{H^r_{0,\sigma}}
    \nonumber&\lesssim M\varepsilon^2.
\end{align}
where $d_{\operatorname{min}} := \min\limits_m |x - z_m| > 0$ (fixed, since \(x \notin \overline{D}\)) is the minimal observation distance, and \(u^{\mathrm{sc},N}\) is the scattered field after \(N\) interactions between the bubbles, defined by
\begin{align}
    u^{\textit{sc},N}(x,t)
    =-\sum_{m=1}^M\frac{C_m}{4\pi|x-\mathbf{z}_m|}
    \sum_{n=0}^N (-1)^n
    (K^{*n}*V)_m\!\big(t-c_0^{-1}|x-\mathbf{z}_m|\big).
\end{align}
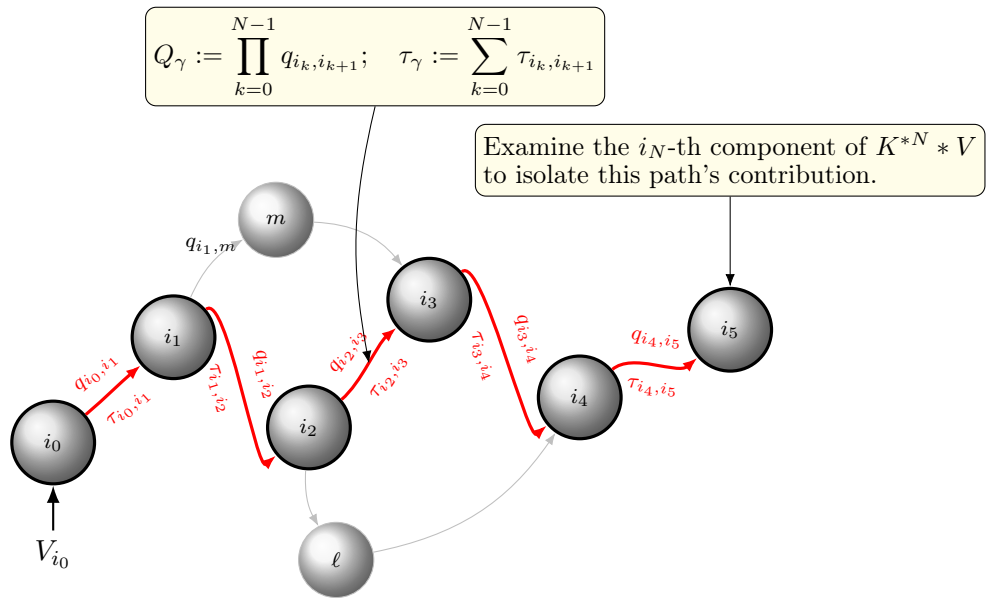
\begin{figure}[H]
    \centering
\begin{tikzpicture}[
  >=Latex,
  sphere/.style={
    circle, draw, very thick, minimum size=11mm, inner sep=0pt,
    shading=ball, ball color=gray!25
  },
  sphereGhost/.style={
    circle, draw=gray!50, minimum size=10mm, inner sep=0pt,
    shading=ball, ball color=gray!10
  },
  pathedge/.style={line width=1.2pt, -{Latex[length=2mm]}, red},
  otheredge/.style={-{Latex[length=1.6mm]}, draw=gray!50},
  note/.style={rectangle, rounded corners, draw, fill=yellow!10, font=\small, inner sep=3pt}
]

\def\nodelist{
  0/0.0/0.0,
  1/1.6/1.4,
  2/3.4/0.2,
  3/5/1.9,
  4/7.0/0.6,
  5/9/1.5
}
\def\N{5}

\foreach \k/\x/\y in \nodelist {
  \node[sphere] (i\k) at (\x,\y) {\(\scriptstyle i_{\k}\)};
}

\def\cx{0.7} \def\cy{0.6}
\foreach \k in {0,...,20}{
  \ifnum\k<\N
    \pgfmathtruncatemacro{\kp}{\k+1}
    \draw[pathedge]
      (i\k) .. controls ($(i\k)+(\cx,\cy)$) and ($(i\kp)+(-\cx,-\cy)$) ..
      node[midway, above, sloped, yshift=2pt, font=\scriptsize, text=red]
        {$q_{i_\k,i_{\kp}}$}
      node[midway, below, sloped, yshift=-2pt, font=\scriptsize, text=red]
        {$\tau_{i_\k,i_{\kp}}$}
      (i\kp);
  \fi
}

\node[sphereGhost, above right=8mm and 6mm of i1] (m) {\(\scriptstyle m\)};
\draw[otheredge] (i1) to[bend left=18] node[above, font=\scriptsize, midway] {$q_{i_1,m}$} (m);
\draw[otheredge] (m) to[bend left=24] (i3);

\node[sphereGhost, below right=10mm and -4mm of i2] (ell) {\(\scriptstyle \ell\)};
\draw[otheredge] (i2) to[bend right=16] (ell);
\draw[otheredge] (ell) to[bend right=22] (i4);

\node[below=6mm of i0] (Vstart) {$V_{i_0}$};
\draw[->, thick] (Vstart) -- (i0.south);

\node[note, above=12mm of i\N, align=left] (endnote) {Examine the $i_N$-th component of $K^{*N}*V$\\ to isolate this path's contribution.};
\draw[->] (endnote.south) -- ++(0,-2mm) -| (i\N.north);

\node[note, align=left] (Qnote) at ($(i2)!0.55!(i3)+(0,4)$) {$
  \displaystyle Q_\gamma:=\prod_{k=0}^{N-1} q_{i_k,i_{k+1}};\quad \tau_\gamma := \sum\limits_{k=0}^{N-1} \tau_{i_k, i_{k+1}}
$};
\draw[->] (Qnote.south) to[bend right=12] ($(i2)!0.5!(i3)$);

\node[below=20mm of i2, align=center, font=\footnotesize] {
  Selected path $\gamma=(i_0,i_1,\ldots,i_N)$ (red) among many possible scattering paths (gray).
};

\end{tikzpicture}
\end{figure}
\noindent
We fix a path of indices $\gamma = (i_0, i_1, \ldots, i_N)$ of length $N$ and denote the product of the scattering coefficients along this path as $Q_\gamma := \prod\limits_{k=0}^{N-1} q_{i_k, i_{k+1}} .$ Corresponding to the path $\gamma$, we assume there exist $p\in (0,1)$ and constants $d_{i_k i_{k+1}}^{(0)}>0,$ independent on $\varepsilon,$ such that $|z_{i_k} - z_{i_{k+1}}| = d_{i_k i_{k+1}}^{(0)} \varepsilon^p$ for $k=0,1,\ldots,N-1.$
\newline
\item Furthermore, assume that $V_{i_0}\in H_{0,\sigma}^r (\Re s\ge\sigma_0>\frac{1}{\omega_M})$ has spectral energy in the resonance band $\Omega_{\operatorname{res}} := [\omega_{\operatorname{res}} - h, \omega_{\operatorname{res}} + h]$ such that its Laplace transform \footnote{Here, $V_{i_0}(\cdot)$ has Laplace transform $\widehat{V}_{i_0}(s) = \frac{s^2}{\omega_M^2  s^2+1}\widehat{\mathbcal{F}}(s)\; \text{with}\ \Re s\ge\sigma_0 >\frac{1}{\omega_M}.$
We also observe that for the resonant angular frequency $\omega_{\operatorname{res}}:= \frac{1}{\omega_M}$, the denominator of the term $\widehat{V}_{i_0}(\cdot)$ vanishes on the imaginary axis $s=\pm i \omega_{\operatorname{res}}$.} satisfies
\begin{align}
|\widehat{V}_{i_0}(\sigma + i \omega)| \ge m \quad \text{for a.e.}\ \omega \in \Omega_{\operatorname{res}},
\end{align}
with constants $m, h > 0$, where $\omega_{\operatorname{res}} := 1/\omega_M$. Then, under the condition
\begin{align}
    p>1-\frac{1}{N},
\end{align}
and we choose the parameters $m$ and $h$ corresponding to the resonance band such that for $\theta \in (0,1)$
\begin{align}
    m\sqrt{2h}>(1-\theta)\big \Vert \mathbcal{F}(z_i,\cdot)\big\Vert_{H^r_{0,\sigma}} \dfrac{d_{\operatorname{max}}e^{\sigma\big(\tau_{\gamma^*}+c_0^{-1}(d_{\operatorname{max}}-d_{\operatorname{min}})\big)}}{d_{\operatorname{min}}Q^{(0)}_{\gamma^*} r_{\operatorname{min}}}\Big(\frac{C^{(0)}_j}{\widetilde{d}}\Big)^{N+1} \omega_M^{2N}\Big(\dfrac{(M-1)\sigma_0^2}{\omega_M^2\sigma_0^2-1}\Big)^{N+2},
\end{align}
the following estimate holds for $r\in \mathbb{Z}_{\ge 0}$:
\begin{align}
    \|u^{\textit{sc},N}(x,t)-u^{\textit{sc},N-1}(x,t)\|_{H^r_{0,\sigma}}
\ge (L^*-(M-1)B)\varepsilon^{N(1-p)+1} \gg M\varepsilon^2.
\end{align}
Here, we denote $L_\gamma := \Bigg( \frac{Q^{(0)}_\gamma C^{(0)}}{4\pi d_{\operatorname{max}}\omega_M^{2N}}\Bigg) e^{-\sigma\big(\tau_\gamma+c_0^{-1}d_{\operatorname{max}}\big)} m \sqrt{2h}\ r_{\operatorname{min}},$
and 
\newline
$B:= \frac{C^{(0)}}{4\pi d_{\operatorname{min}}}\frac{\sigma_0^2}{\omega_M^2\sigma_0^2-1} \alpha_\infty^{N+1}e^{-\sigma c_0^{-1}d_{\operatorname{min}}}\|\mathbcal{F}\|_{H_{0,\sigma}^r},$ choosing any maximizer $\gamma^*$ such that $\gamma^* \in \arg \max\limits_{\gamma\in \Gamma_N} L_\gamma,$ with $L_{\gamma^*}= L^*$. Here, $Q_\gamma^{(0)}$ is a positive geometric constants independent of $\varepsilon$, $\tau_\gamma := \sum\limits_{k=0}^{N-1} \tau_{i_k, i_{k+1}}$, $r_{\operatorname{min}}:= \big(\min\limits_{\omega\in \Omega_{\operatorname{res}}}(\sigma^2+\omega^2)^r\big)^\frac{1}{2}$, $d_{\operatorname{max}} := \max\limits_m |x - z_m| > 0$ (fixed, since \(x \notin \overline{D}\)) is the maximal observation distance, and $\Gamma_N$ as the collection of all such ordered paths of length $N.$ 
\end{enumerate}
\end{theorem}

\subsection{Feynman-diagram viewpoint and relation to acoustic multiple scattering}\label{Feynman}

\noindent
We begin by recalling the schematic representation outlined by Nelson \cite{Nelison} on the concept of Feynman diagram. A Feynman diagram—with \( N \)-interactions is given by
\begin{align}
    \big(\text{Coupling Constants}\big)\int \big(\text{Propagators}\big)\; \big(\text{Delta functions}\big)\; d\big(\text{internal lines}\big).
\end{align}
In this work, we reduce the acoustic scattering problem by a collection of bubbles to a delay-coupled hyperbolic system for amplitudes \(\bm{Y}(t) = (Y_i)_{i=1}^M\) of the form
    \begin{align} \label{matrix2}
        \begin{cases}
              \mathcal{A} \frac{d^2}{dt^2}\bm{Y}(t) + \bm{Y}(t) = \bm{\mathcal{F}}(t) \quad \text{in } (0, T), \\
              \bm{Y}(0) = \frac{d}{dt} \bm{Y}(0) = \mathbf{0}.
        \end{cases}
    \end{align}
This leads to a Volterra-type integral equation for \(\bm{Y}(t)\), which has the solution
\begin{align}
    \bm{Y}(t) = \sum_{n=0}^\infty (-1)^n K^{*n} * \bm{\mathcal{F}}(t).
\end{align}
Here, \( K := \big[K_{ij}\big] \) is a causal matrix kernel, and the expression above is a Neumann series in the time-domain convolution. Each \( K_{ij} \) consists of two parts, namely:
\[
K_{ij}(t)= K_{ij}^{(d)}(t) + K_{ij}^{(c)}(t),
\]
where \( K_{ij}^{(d)}(t) \) is the atomic part,
\[
K_{ij}^{(d)}(t):= \frac{q_{ij}}{\omega_M^2}\delta(t-\tau_{ij}),
\]
and \( K_{ij}^{(c)}(t) \) represents the non-atomic part, which capture the effects of dispersion. Furthermore, the following identity holds in the distributional sense (see (\ref{kd})):
    \begin{align}
         \big(K^{(d)}\big)^{*N}_{ij}(t)
         = \frac{1}{\omega_M^{2N}}\sum\limits_{\substack{\gamma=(i_0,i_1,\ldots,i_N) \\i_0=i,i_N=j}} \Bigg(\prod\limits_{k=0}^{N-1} q_{i_k,i_{k+1}}\Bigg) \delta \Bigg(t-\sum_{k=0}^{N-1}\tau_{i_k i_{k+1}}\Bigg),
    \end{align}
where \(\gamma\) denotes a path of length \(N\) from a bubble location \(z_i\) to another one \(z_j\).
\noindent
In our setting, for a chosen path of indices \(\gamma = (i_0, i_1, \ldots, i_N)\) of length \(N\), the indices \(i_1, i_2, \ldots, i_{N-1}\) represent the vertices of the diagram. The Dirac part of the kernel, \(K_{ij}^{(d)}(t):= \frac{q_{ij}}{\omega_M^2}\delta(t-\tau_{ij}),\)  plays the role of a propagator-vertex pair. More precisely, each Dirac kernel
\[
K_{i_k i_{k+1}}^{(d)}(t):= \frac{q_{i_k i_{k+1}}}{\omega_M^2}\delta(t-\tau_{i_k i_{k+1}})
\]
acts as a retarded point-propagator (analogous to the Helmholtz-type propagator in standard quantum field theory (QFT); see \cite[Chapter 4]{Nair}), propagating a singularity from bubble \(i_k\) to bubble \(i_{k+1}\) with a fixed travel time \(\tau_{i_k i_{k+1}}\), while $\prod\limits_{k=0}^{N-1} q_{i_k,i_{k+1}}$ represents the product of coupling constants. Observing the form of $\big(K^{(d)}\big)^{*N}_{ij}(t)$, we see that in $\mathcal{D}'(0,\infty)$, the convolution over the intermediate interaction times $t_1,t_2,\ldots,t_{N-1}$,
$$\big(K^{(d)}\big)^{*N}_{ij}(t) = \int_{0<t_{N-1}<\ldots<t_1<t}\prod_{k=0}^{N-1}K_{i_k i_{k+1}}^{(d)}(t_k-t_{k+1})dt_1\ldots dt_{N-1},\; t_0:=t,t_N:=0,$$
reduces, via a distributional identity, to the Dirac delta
$$\delta \Bigg(t-\sum_{k=0}^{N-1}\tau_{i_k i_{k+1}}\Bigg).$$
This delta function is the analogue of the kinematic conservation-enforcing delta functions in Nelson's schematic formula. To this end, choosing a fixed path $\gamma=(i_0,i_1,\ldots,i_N)$ with $i_0=i$ and $i_N=j$ and integrating over the time-ordered simplex $\{0<t_{N-1}<\cdots<t_1<t\}$ using the conventions $t_0:=t$ and $t_N:=0$, the contribution of this path to the N-fold convolution is
\begin{align}
    \big(\text{Coupling Constants}\big)\int \big(\text{Propagators}\big)\; \big(\text{Delta functions}\big)\; d\big(\text{internal lines}\big).
\end{align}
which in our setting takes the explicit form
\begin{align}
    \Bigg(\prod_{k=0}^{N-1}\frac{q_{i_k i_{k+1}}}{\omega_M^{2}}\Bigg)
    \int_{0<t_{N-1}<\cdots<t_1<t}
    \prod_{k=0}^{N-1}
     \delta\bigl(t_k - t_{k+1} - \tau_{i_k i_{k+1}}\bigr)\
    dt_1\cdots dt_{N-1}.
\end{align}
It is important to note that the Dirac part of the kernel plays a dual role: the individual $K_{i_k i_{k+1}}^{(d)}$ act as propagators, while their convolution along a path creates a delta constraint that eliminates the internal time variables.

\subsection{Discussion}

The results displayed in Section \ref{Main-results} and Section \ref{Feynman} establish a rigorous and interpretable framework for transient multiple scattering in dispersive bubbly media, addressing critical gaps in classical Foldy-Lax theory while introducing a novel diagrammatic perspective. Below, we discuss these results and highlight their potential implications.
\begin{enumerate}

\item  Theorem \ref{maintheorem}’s solvability and error estimates resolve a longstanding limitation of time-domain Foldy-Lax models: the lack of rigorous justification for dispersion-inclusive systems. By combining anisotropic Hilbert spaces (\(H_{0,\sigma}^r\)) and Hardy-Sobolev spaces, we prove that the delayed-coupled hyperbolic system admits a unique solution expressed as a convergent Neumann series of convolution powers. This convergence is non-trivial—unlike frequency-domain models, time-domain dispersive systems introduce delayed interactions and frequency-dependent scattering coefficients (via Minnaert resonance) that violate classical compactness assumptions. Our use of Laplace transform isometries and operator-valued function space techniques ensures the scattering kernel’s boundedness, a key technical contribution that validates the series’ convergence for practical parameter ranges.The truncation error’s geometric decay (Section \ref{Main-results}, Result 2) is particularly impactful for applications. Near resonance, where dispersive effects dominate, higher-order multi-scattering (up to \(N < 1/(1-p)\)) contributes meaningfully to transient fields, while measurement errors (of order \(\varepsilon^2\)) are dominated by these higher-order terms. This quantifies the Foldy-Lax approximation’s accuracy: for closely spaced bubbles (\(p < 1\)) and incident waves with spectral energy in the resonance band, truncating the series at N interactions retains physically relevant dynamics without spurious artifacts. This resolves ambiguity in prior works, which often assumed truncation without quantifying how many interactions are necessary to capture transient responses.

\item  The mapping between Foldy-Lax amplitude equations and Feynman diagrams (Section \ref{Feynman}) provides a powerful conceptual and computational tool, bridging classical multiple scattering theory with quantum field theory (QFT) techniques. By interpreting inter-bubble scattering events as vertices, wave propagation as retarded propagators, and convolution powers as multi-scattering paths, we reduce the complexity of higher-order interactions to diagrammatic kinematics. This correspondence is not merely heuristic: the Dirac delta functions enforcing time-delay conservation in the convolution kernels directly mirror the kinematic constraints in QFT Feynman rules, linking classical acoustic scattering to fundamental principles of particle interaction modeling. Practically, this diagrammatic framework simplifies the tracking of multi-scattering paths. For example, a path \(A \to B \to C\) (two vertices) corresponds to a Feynman diagram whose amplitude scales with \(\varepsilon^{2(1-p)+1}\), aligning with Theorem \ref{maintheorem}’s error estimates. Additionally, the diagrammatic perspective clarifies why resonance enhances higher-order contributions: near Minnaert resonance, the scattering kernel’s spectral amplitude is amplified, making diagrammatic terms corresponding to longer paths (more vertices) significant.


\item  These results extend beyond bubbly acoustic media to any dispersive system with resonant scatterers. Indeed, the functional-analytic framework can be adapted to other dispersive mechanisms (e.g., Lorentz resonance in dielectrics) by modifying the scattering kernel’s frequency dependence. 
\end{enumerate}
\section{Proof of Theorem \ref{maintheorem}}

\noindent 
We first state the following proposition.
\begin{proposition}\cite[Theorem 1.1]{Arpan-Sini-JEEQ} \label{main-prop}
    Consider the acoustic problem (\ref{hyperbolic}) generated by a cluster of resonating acoustic gas bubbles \( D_j \) for \( j = 1, 2, \ldots, M \). Then, under the following conditions:
\begin{align}\label{inversion-cond}
    \frac{\rho_\mathrm{c}}{4\pi} \, \text{vol}(\mathrm{B}_j) \, \Big(\frac{\varepsilon}{d}\Big)^6 \, \Big(\frac{1}{\lambda_1^{(3)}}\Big)^2 < 1,\ \textcolor{black}{j=1,2,\ldots,M}, 
    \quad \text{and} \quad 
    \mathbf{C} \, \max_{1 \leq m \leq M} \sum_{\substack{j=1 \\ j \neq m}}^M \frac{1}{4\pi |z_m - z_j|} < \omega_M^2,
\end{align}
with \( \mathbf{C} := \max\limits_{1 \leq j \leq M} \mathbf{C}_j \), the scattered field \( u^\textit{sc}(\mathrm{x}, t) \) has the following asymptotic expansion:
\begin{align} \label{assymptotic-expansion-us}
    u^\textit{sc}(\mathrm{x}, t) = -\sum_{m=1}^M \frac{C_m}{4\pi |\mathrm{x} - \mathrm{z}_m|} \mathrm{Y}_m\big(t - \mathrm{c}_0^{-1} |\mathrm{x} - \mathrm{z}_m|\big) + \mathcal{O}(M\varepsilon^{2}) \quad \text{as} \quad \varepsilon \to 0,
\end{align}
for \( (\mathrm{x}, t) \in \mathbb{R}^3 \setminus \mathcal{K} \times (0, \mathrm{T}) \) with \( \overline{\mathbf{\Omega}} \subset \subset \mathcal{K} \), \textcolor{black}{($\mathcal{K}$ is any bounded domain such that $\operatorname{dist}(\partial \mathcal{K}, \bm{\Omega}) > \operatorname{const}$, where $\operatorname{const}$ is independent of $\varepsilon$)}, where \( \big(\mathrm{Y}_j\big)_{j=1}^M \) is the vector solution to the following non-homogeneous second-order matrix differential equation with zero initial conditions: 
\begin{align} \label{matrixmulti}
    \begin{cases}
        \omega_M^2 \frac{\mathrm{d}^2}{\mathrm{d} t^2} \mathrm{Y}_m(t) + \mathrm{Y}_m(t) + \sum\limits_{\substack{j=1 \\ j \neq m}}^M \frac{\mathbf{C}_j}{4\pi |z_m - z_j|} \frac{\mathrm{d}^2}{\mathrm{d} t^2} \mathrm{Y}_j\big(t - \mathrm{c}_0^{-1} |\mathrm{z}_m - \mathrm{z}_j|\big) = \frac{\partial^2}{\partial t^2} u^\textit{in}(z_m,t), &\textit{in} \; (0, \mathrm{T}), \\
        \mathrm{Y}_m(0) = \frac{\mathrm{d}}{\mathrm{d} t} \mathrm{Y}_m(0) = 0,
    \end{cases}
\end{align}
where \( \omega_M^2 = \frac{\rho_c}{2 \overline{\kappa}_{\mathrm{b},m}} \Lambda_{\partial B_m} \) represents the inverse square of the Minnaert frequency \( \omega_{\operatorname{Min}}:=\sqrt{\frac{2 \overline{\kappa}_{\mathrm{b},m}}{\rho_c \Lambda_{\partial B_m}} }\) of the bubble \( D_m \) (see, for instance, \cite{AZ18, DGS-21}), and \( \mathbf{C}_j := C^{(0)}_j \varepsilon \) with \( C^{(0)}_j := \operatorname{vol}(B_j) \frac{\rho_c}{\overline{\kappa}_{\mathrm{b},j}} \).
\\
Here, \( \displaystyle \Lambda_{\partial B_m} := \frac{1}{|\partial B_m|} \int_{\partial B_m} \int_{\partial B_m} \frac{(\mathrm{x} - \mathrm{y}) \cdot \nu_\mathrm{x}}{|\mathrm{x} - \mathrm{y}|} \, d\sigma_\mathrm{x} \, d\sigma_\mathrm{y} \) is a geometric constant. The well-posedness of the system of differential equations (\ref{matrixmulti}) is discussed in \cite[Section 2.4]{Arpan-Sini-JEEQ}.
\end{proposition}
\noindent
\begin{lemma}
The system of differential equations (\ref{matrixmulti}) reduces, under the Fourier–Laplace transform, to the following matrix equation:
\begin{align}
    (I+T(s))\widehat{\mathbf Y}(s) = V(s), \qquad \Re s > \sigma_0,
\end{align}
where
\begin{align}
    D(s) = \operatorname{diag}(\omega_M^2 s^2+1)_{i=1}^M,
\qquad Q(s) = (q_{ij}e^{-s\tau_{ij}})_{i,j=1}^M,
\end{align}
and
\begin{align}
    T(s) := D(s)^{-1}s^2 Q(s),
\qquad V(s) := D(s)^{-1}s^2 F(s),
\end{align}
with \(\mathbcal{F}(s) = \Laplace [\mathbcal{F}(t)]\) the Laplace transform of the causal forcing vector
$\mathcal{F}_i(t) = \partial_t^2 u^{\mathrm{in}}(z_i,t), \; i=1,\dots,M,\quad M \in \mathbb N.$
Moreover, \(V(s)\) is analytic for \(\Re s > \sigma_0\).
For each i=1,\dots,M, the parameters satisfy
\begin{align*}
    \omega_M^2 > 0,
\qquad \tau_{ij} = c_0^{-1}|z_i - z_j| \ge 0 \ (j=1,\dots,M),
\qquad q_{ij} = \frac{\mathbf C_j}{4\pi |z_i - z_j|} \in \mathbb R,
\end{align*}
with the convention \(q_{ii} = 0\).
\end{lemma}

\begin{proof}
  Each diagonal entry of \(D(s)\) is of the form \(\omega_M^2 s^2 + 1\). This polynomial has its zeros on the imaginary axis at $s = \pm \tfrac{i}{\omega_M}.$ Therefore, for every fixed \(\omega_M^2 > 0\), the scalar function $s \mapsto \frac{1}{\omega_M^2 s^2 + 1}$ is analytic in the open half-plane \({ \Re s > \sigma_0 }\). Consequently, the diagonal matrix \(D(s)^{-1}\) is analytic entry-wise in \({ \Re s > \sigma_0 }\). Next, note that
\[
T(s) = D(s)^{-1} s^2 Q(s).
\]
Here, both \(s \mapsto s^2\) and \(s \mapsto e^{-s\tau_{ij}}\) are entire functions, and their product with the analytic factor \((\omega_M^2 s^2 + 1)^{-1}\) remains analytic for \(\Re s > \sigma_0\). Explicitly, each entry is given by
\[
T_{ij}(s) = \frac{s^2}{\omega_M^2 s^2 + 1} q_{ij} e^{-s\tau_{ij}}.
\]
This is the product of two analytic functions on \({\Re s > \sigma_0}\), hence analytic there. Therefore \(T(s)\) is analytic in \({ \Re s > \sigma_0 }\).
\newline
A similar argument applies to
\[
V(s) = D(s)^{-1} s^2 \mathcal{F}(s),
\]
since the Laplace transform (F(s)) of a causal forcing function is analytic for \(\Re s > \sigma_0\).
\newline
Thus, both \(T(s)\) and \(V(s)\) are analytic in the right half-plane, and the claimed matrix equation follows.  
\end{proof}
\noindent
Next, the crucial step is to provide a justification that the inverse Laplace of each scalar entry of $T(s)$ i.e.
$$T_{ij}(s) = \frac{s^2}{\omega_M^2 s^2 + 1} q_{ij} e^{-s\tau_{ij}}$$ exist.
\newline
We start with denoting the following term as $R_{\omega_M^2}(s) := \frac{s^2}{\omega_M^2 s^2 + 1},$ which has the following precise decomposition for $\omega_M^2>0$:
\begin{align*}
    R_{\omega_M^2}(s) = \frac{1}{\omega_M^2} - \frac{1}{\omega_M^2\big(\omega_M^2 s^2 + 1\big)}
\end{align*}
We have that \(\mathcal L^{-1}[\frac{1}{\omega_M^2}]=\frac{1}{\omega_M^2}\delta(t)\) by linearity and \(\mathcal L^{-1}[1]=\delta\). Also, we have \(\mathcal L^{-1}[(\omega_M^2 s^2+1)^{-1}]=(\frac{1}{\omega_M})\sin(\frac{t}{\omega_M})\mathbf1_{t\ge0}\), as noted above; multiply by \(-\frac{1}{\omega_M^2}\) to obtain the second piece:
\[
\mathcal L^{-1}\Big[-\frac{1}{\omega_M^2}\cdot\frac{1}{\omega_M^2 s^2+1}\Big](t)
= -\frac{1}{\omega_M^3}\sin\!\Big(\frac{t}{\omega_M}\Big)\mathbf1_{t\ge0}.
\]
\noindent
Thus for the rational factor alone,
\begin{align}\label{decom}
    \mathcal L^{-1}[R_{\omega_M^2}](t)=\frac{1}{\omega_M^2}\delta(t)-\frac{1}{\omega_M^3}\sin\!\Big(\frac{t}{\omega_M}\Big)\mathbf1_{t\ge0},
\end{align}
as distributions supported in \([0,\infty)\).
\newline
Let $f\in\mathcal D'([0,\infty))$ be a distribution supported in $[0,\infty)$ and Laplace transformable. Then, due to the well-known shift theorem, we have the following property
$$\Laplace \{f(t-a)\mathbf1_{t\ge a}\} = e^{-sa}\Laplace\{f\} \; \text{for}\ \Re s > \sigma_0\ \text{and}\ a>0,$$
where $\sigma_0 = \inf\{\sigma: e^{-\sigma t}f(t) \in L^1(0,\infty)\}.$ In addition, we note that as $e^{-st}$ is an entire analytic function, $\Laplace\{f\}$ is a holomorphic function for $\Re s > \sigma_0.$ 
\newline
Consequently, using the above property to the decomposition in (\ref{decom}) yields the explicit inverse for \(T_{ij}\)
\begin{align}\label{kernelt}
K_{ij}(t)=\mathcal L^{-1}[T_{ij}](t)
= q_{ij}\Big(\frac{1}{\omega_M^2}\delta(t-\tau_{ij})
-\frac{1}{\omega_M^3}\sin\!\big(\tfrac{t-\tau_{ij}}{\omega_M}\big)\mathbf1_{t\ge\tau_{ij}}\Big).
\end{align}

\begin{lemma}\label{imlemma}
Let us fix $\omega_M>0$ and $\sigma_0>\frac{1}{\omega_M}$. Then, the Neumann series for \( (I + T(s))^{-1} \) converges uniformly on the half plane \( \{s\in \mathbb{C}: \Re s \ge \sigma_0\} \) if the bubbles are distributed in such a way that the following inequality holds:
\[
\varepsilon \cdot \max_{1 \leq i \leq M} \sum_{j \neq i} \frac{C^{(0)}_j}{4\pi|z_i - z_j|} < \omega_M^2 \left(1 - \frac{1}{\omega_M^2\sigma_0^2}\right).
\]
\end{lemma}
\begin{proof}
First, we have:
\[
T(s) = D(s)^{-1} s^2 Q(s),
\]
where
\[
D(s) = \operatorname{diag}(\omega_M^2 s^2 + 1)_{i=1}^M, \quad Q(s) = (q_{ij} e^{-s \tau_{ij}})_{i,j=1}^M,
\]
and
\[
q_{ij} = \frac{\mathbf{C}_j}{4\pi |z_i - z_j|}, \quad \tau_{ij} = c_0^{-1} |z_i - z_j|.
\]
Also, \(\mathbf{C}_j = C^{(0)}_j \varepsilon\) with $\varepsilon\ll 1.$
\newline
First, we will show that for all \(s \in \mathbb{C}\) with \(\Re s\ge \sigma_0 > \frac{1}{\omega_M}\), we have the uniform bound:
\[
\sup\limits_{\Re s\ge \sigma_0}\left|\frac{s^2}{\omega_M^2 s^2 + 1}\right| \leq \frac{\sigma_0^2}{\omega_M^2 \sigma_0^2 - 1} = \frac{1}{\omega_M^2} \cdot \frac{1}{1 - \frac{1}{\omega_M^2 \sigma_0^2}}
\]
\newline
Let us now recall that
\begin{align}\label{defrs}
R_{\omega_M^2}(s) = \frac{s^2}{\omega_M^2 s^2 + 1}.
\end{align}
We then begin by providing a uniform bound for $R_{\omega_M^2}(s)$ in terms of $s.$ Since $|R(s)|\to \infty$ near the poles $\pm \frac{i}{\omega_M}$, we require \( \Re s\sigma_0 \). Moreover, we can still approach arbitrarily close to these poles, making a uniform bound impossible over the entire right half-plane. Therefore, we restrict our analysis on the half-plane 
\[
\Re s\ge \sigma_0 > \frac{1}{\omega_M}.
\]
\newline
Next, due to the consideration $\Re s\ge \sigma_0 > \frac{1}{\omega_M},$ we see that
\begin{align}\label{in}
    |\omega_M^2 s^2 + 1| \geq \big| |\omega_M^2s|^2 - 1\big| \geq \omega_M^2|s|^2 - 1 \geq \omega_M^2\sigma_0^2 - 1 > 0.
\end{align}
Since, we have
\[
|s| \geq \sigma_0 > \frac{1}{\omega_M},\ \text{which further implies}\ \omega_M^2 |s|^2 > 1.
\]
It shows the validity of the inequality (\ref{in}). Consequently, we obtain 
\begin{align}\label{2.10}
    \sup\limits_{\Re s\ge \sigma_0}|R(s)| = \frac{|s|^2}{|\omega_M^2 s^2 + 1|} \leq \frac{|s|^2}{\omega_M^2|s|^2 - 1} \leq \frac{1}{\omega_M^2} \cdot \frac{1}{1 - \frac{1}{\omega_M^2|s|^2}} \leq \frac{1}{\omega_M^2} \cdot \frac{1}{1 - \frac{1}{\omega_M^2\sigma_0^2}}
\end{align}
In addition to that, due to the fact \( \Re s \ge \sigma_0 \) and \( \tau_{ij} \geq 0 \), we have \( |e^{-s\tau_{ij}}| \leq 1 \). Subsequently, the Neumann series converges whenever
\begin{align}
    \sup\limits_{\Re s\ge \sigma_0}\|T(s)\|_{\operatorname{op}} \leq \frac{\sigma_0^2}{\omega_M^2 \sigma_0^2 - 1} \cdot \max_{1 \leq i \leq M} \sum_{j \neq i} |q_{ij}| < 1.
\end{align}
Here, we note that the operator norm $\Vert \cdot \Vert_{\operatorname{op}}$ on $\mathbb{C}^{M\times M}$ is the $\ell^\infty-$induced norm, given by $\Vert A \Vert_{\operatorname{op}}=\max\limits_{i}\sum\limits_{j}|A_{ij}|$. We use this norm in the subsequent analysis.
\newline
We also recall that $q_{ij}=\varepsilon\frac{C_{j}^{(0)}}{4\pi d_{ij}} \asymp \varepsilon^{1-p}\frac{C_{j}^{(0)}}{4\pi \widetilde{d}},$ where we set $q^{(0)}_{ij} \asymp \frac{C_{j}^{(0)}}{4\pi\widetilde{d}}.$
Substituting \(q_{ij} = \varepsilon^{1-p}q^{(0)}_{ij}\), we obtain the convergence condition of the Neumann series
\[
\varepsilon^{1-p} \cdot \max_{1 \leq i \leq M} \sum_{j \neq i}  q^{(0)}_{ij} < \frac{\omega_M^2 \sigma_0^2 - 1}{\sigma_0^2} = \omega_M^2 \left(1 - \frac{1}{\omega_M^2\sigma_0^2}\right),
\]
which completes the proof.
\end{proof}
\noindent
Applying the above Lemma to \((I+T(s))\widehat{\mathbf Y}(s)=V(s)\) gives the unique solution
\begin{align}\label{Neumann-Frequency}
    \widehat{\mathbf{\bm{Y}}}(s) = (I+T(s))^{-1}V(s)
= \sum_{n=0}^\infty (-1)^n T(s)^{n}V(s).
\end{align}
Next, we aim to provide a time-domain expression $\mathrm{\mathbf{Y}}(t)$ corresponding to the above frequency domain expression $\widehat{\mathbf{\bm{Y}}}(s).$ Therefore, we next prove that the inverse Laplace of \(T(s)^n\) exists and equals \(K^{*n}\) \big(\text{i.e.}\ \(\mathcal L[K^{*n}]=T(s)^n\big).\) 
\newline
Before proceeding, we state the following Lemma:
\begin{lemma}\label{lemma-convolution}\cite{borel}
Let \(\mu\) be a finite Borel measure on \([0,T]\) and \(f\in L^1([0,T])\). Then
$$(\mu * f)(t)= \int_0^Tf(t-\tau)d\mu(\tau)\ \text{exists for a.e.}\ t,$$ and satisfies
$$\|\mu*f\|_{L^1}\le \|\mu\|_{TV}\,\|f\|_{L^1}.$$
\end{lemma}
\noindent
The following proposition will prove that for all $n$, the n-fold convolution power \(K^{*n}\) is well defined and has entries in the function space $\mathcal{B}([0,T]) \oplus L^1([0,T]).$
\begin{proposition}\label{prop1}
Let $K$ be the matrix-valued distribution with entries  
\begin{align}\label{matrix K1}
K_{ij}(t)=\mathcal L^{-1}[T_{ij}](t)
= q_{ij}\Big(\frac{1}{\omega_M^2}\delta(t-\tau_{ij})
-\frac{1}{\omega_M^3}\sin\!\big(\tfrac{t-\tau_{ij}}{\omega_M}\big)\mathbf1_{t\ge\tau_{ij}}\Big),
\end{align}
and $T(s)$ be the matrix valued function defined by
\begin{align}
    T_{ij}(s) = \frac{s^2}{\omega_M^2 s^2 + 1} q_{ij} e^{-s\tau_{ij}}.
\end{align}
Let us now define the convolution powers recursively as follows $K^{*0}=\delta I,$ $K^{*1}= K,$ and $K^{*(n+1)}= K^{*n} * K$ for $n\ge 1.$ Then, we have each entry of $K^{*n}$ is an element of $\mathcal{B}([0,T]) \oplus L^1([0,T])$ and the Laplace transform of $K^{*n}$ satisfies \(\mathcal L[K^{*n}]=T(s)^n.\)

\end{proposition}

\begin{proof}
We prove by induction on \(n\).
\newline
For $n=0,$ we know by definition that \(K^{*0}=\delta I\). Since, Dirac delta function $\delta$ is a finite measure on $[0,T]$ with total variation 1, it follows that $\delta I \in \mathcal{B}([0,T])^{M\times M},$ and thus in $\big(\mathcal{B}([0,T]) \oplus L^1([0,T])\big)^{M\times M}.$ Therefore, its Laplace transform is $\mathcal{L} [\delta I] = I = T(s)^0.$
\newline
Next, we assume that the statement holds for all $n\ge 0.$ Then, we have $K^{*(n+1)}=K^{*n}* K$. Now, by induction hypothesis $K^{*n} \in \mathcal{B}([0,T]) \oplus L^1([0,T])$, and by definition of $K$ (\ref{matrix K1}), which is combination of a measure and $L^1$ function, we see that it also lies in $\mathcal{B}([0,T]) \oplus L^1([0,T]).$ Therefore, due to Lemma \ref{lemma-convolution} for convolution, we find that each entry of $K^{*(n+1)}$ is in $\mathcal{B}([0,T]) \oplus L^1([0,T]).$
\newline
For the Laplace transform, using convolution theorem, we have the following:
\begin{align}
    \mathcal L[K^{*(n+1)}]= \mathcal L[K^{*n}*K]= \mathcal L[K^{*n}]\cdot \mathcal L[K] = T(s)^n\cdot T(s) =T(s)^{n+1}.
\end{align}
This completes the induction.
\end{proof}
\noindent
Before moving forward, we point out that we will not explicitly use the norm properties of the Banach space $\mathcal{B}([0,T]) \oplus L^1([0,T])$. Instead, we only use the fact that it is a Banach space which is stable under convolution and compatible with the Laplace transform, in the sense that \(\mathcal L[K^{*n}]=T(s)^n.\)

\noindent
Consequently, Proposition \ref{prop1} shows that the Neumann series in the time domain, $\sum\limits_{n=0}^\infty (-1)^n K^{*n}*V(t)$ is well-defined and  its Laplace transform is the Neumann series in the frequency domain, $\sum\limits_{n=0}^\infty (-1)^n T^{n}(s)V(s),$ (\ref{Neumann-Frequency}), which converges to 
$\widehat{\mathbf{\bm{Y}}}(s) = (I+T(s))^{-1}V(s)$ provided $\sup\limits_{\Re s\ge \sigma_0}\Vert T(s) \Vert_\infty <1.$
\newline
Therefore, we have the following expression in the time-domain
\begin{align}\label{Neumanntime}
    \mathbf{\bm{Y}}(t) = \sum_{n=0}^\infty (-1)^n K^{*n}*V(t),
\end{align}
the partial sums
\begin{align}
    \mathbf{\bm{Y}}^N(t) = \sum_{n=0}^N (-1)^n K^{*n}*V(t),
\end{align}
the remainder
\begin{align}\label{imesti1}
    R_N(t):= \sum_{n=0}^\infty (-1)^n K^{*n}*V(t)-\sum_{n=0}^N (-1)^n K^{*n}*V(t) = \sum_{n=N+1}^\infty (-1)^n K^{*n}*V(t),
\end{align}
and provide an estimate for the remainder term $R_N$ in the following section.

\subsection{The remainder estimate \(\|Y-Y^N\|\)}\label{remainder}
\noindent

\noindent
We connect this to the convolution operators central to our problem. For a causal kernel (matrix-valued) \(K\), the convolution operator
\begin{align}
    (\mathcal{P}_K F)(t):= (K*F)(t) = \int_0^tK(t-\tau)F(\tau)d\tau
\end{align}
has the Laplace transform
\begin{align}
    \Laplace \big[\mathcal{P}_K F\big](s) = \widehat{K}(s)\cdot\widehat{F}(s).
\end{align}
We define the multiplication operator \(M_{\widehat{K}}: H_{r,\sigma}\to H_{r,\sigma}\) by
\begin{align}
   \big( M_{\widehat{K}}F\big)(s) := \widehat{K}(s)\cdot\widehat{F}(s).
\end{align}
Since $\widehat{K}\in H^\infty\big(\mathbb{C}_{\sigma}^+,\mathcal{L}(\mathbb{C}^M)\big)$, for any \(\widehat{F}\in H_{r,\sigma},\) we derive the following estimate due to Corollary \ref{cor1.3}:
\begin{align}
    \Vert M_{\widehat{K}}\widehat{F} \Vert_{H_{r,\sigma}}^2 = \frac{1}{2\pi}\int_{-\infty}^\infty |s|^{2r}\,\|\widehat K(s)\|_{\operatorname{op}}^2\,\|\widehat F(s)\|^2\,d\omega
    \le \Big(\sup_{\Re s=\sigma_0} \|\widehat K(s)\|_{\operatorname{op}}^2\Big)\, \| \widehat F \|_{H_{r,\sigma}}^2.
\end{align}
Since, \(\mathcal{L}\) is an isometry between \(H^r_{0,\sigma}\) and \(H_{r,\sigma},\) and \(\mathcal{P}_K\) corresponds to \( M_{\widehat{K}}\), \Big(\(\text{i.e.,}\ H^r_{0,\sigma} \overset{\Laplace}{\longrightarrow}H_{r,\sigma} \overset{M_{\widehat{K}}}{\longrightarrow}H_{r,\sigma} \overset{\Laplace^{-1}}{\longrightarrow}H^r_{0,\sigma}\Big)\) for \(F\in H^r_{0,\sigma}\), we derive that
\begin{align}
    \|K*F\|_{H^r_{0,\sigma}}
    = \| M_{\widehat K}\widehat F\|_{H_{r,\sigma}}
    \le \Big(\sup_{\Re s=\sigma} \|\widehat K(s)\|_{\operatorname{op}}\Big)\,\|F\|_{H^r_{0,\sigma}} .
\end{align}
Above, we denoted $\widehat{K}(s)$ as the Laplace transform of a causal kernel $K,$ which belong to $\mathcal{B}[0,T]\oplus L^1[0,T].$ In addition, $\widehat{K}(s)$ is analytic and bounded for $\Re s\ge\sigma_0$ with $\sigma_0>\frac{1}{\omega_M}$ as well as it is continuous on the closure $\Re s\ge\sigma_0$. This means that we can apply maximum modulus principle and consequently, for $\sigma\ge \sigma_0,$ we have
\begin{align}
    \sup_{\Re s\ge\sigma} \|\widehat K(s)\|_{\operatorname{op}} = \sup_{\Re s=\sigma} \|\widehat K(s)\|_{\operatorname{op}} .
\end{align}
In particular, we set $\alpha_{\sigma_0}:= \sup\limits_{\Re s=\sigma_0}\|\widehat K(s)\|_{\operatorname{op}}$ as \(\sup\limits_{\Re s\ge\sigma} \|\widehat K(s)\|_{\operatorname{op}}\le \alpha_{\sigma_0.}\) Consequently, we obtain
\begin{align}
    \Vert K*F\Vert_{H^r_{0,\sigma}} \le \alpha_{\sigma_0} \Vert F\Vert_{H^r_{0,\sigma}}.
\end{align}
Hence, from this estimate, by submultiplicativity of the operator norm and induction on $n$, it follows that
\begin{align}\label{imesti}
    \Vert \mathcal{P}_K ^n F \Vert_{H^r_{0,\sigma}} \le \Vert \mathcal{P}_K  \Vert^n \Vert F \Vert_{H^r_{0,\sigma}} = \alpha_{\sigma_0}^n \Vert F \Vert_{H^r_{0,\sigma}},
\end{align}
where \(\mathcal{P}_K \) is the convolution operator \(\mathcal{P}_K : F\to K* F\) acting on the Hilbert space \(H^r_{0,\sigma}(0,\infty;\mathbb{C}^M)\) and \(\alpha_{\sigma_0}:= \Vert \mathcal{P}_K \Vert_{\mathcal{L}(H^r_{0,\sigma})}<\infty.\)
\newline
Now, under the condition \(\alpha_{\sigma_0}<1,\) the geometric series $\sum\limits_{n=0}^\infty \alpha_{\sigma_0}^n$ converges. Then, using the triangle inequality and the estimate \eqref{imesti}, we obtain from the Neumann series representation the following:
\begin{align}
    \|R_N\|_{H^r_{0,\sigma}} \leq \sum_{n=N+1}^{\infty} \|K^{*n}*V\|_{H^r_{0,\sigma}}\le \left(\sum_{n=N+1}^\infty \alpha_{\sigma_0}^{N+1}\right)\|V\|_{H^r_{0,\sigma}} = \frac{\alpha_{\sigma_0}^{N+1}}{1-\alpha_{\sigma_0}}\|V\|_{H^r_{0,\sigma}}.
\end{align}
Due to Lemma \ref{imlemma} and Proposition \ref{prop1}, with $T(s)= \widehat{K}(s)$, it is easy to see that
\begin{align}
    \alpha_{\sigma_0} = \sup_{\Re s={\sigma_0}}\Vert T(s)\Vert_{\operatorname{op}}.
\end{align}
Let us now recall \(q_{ij} = \frac{\mathbf C_j}{4\pi |z_i - z_j|}\) where $C_j= C_j^{(0)}\varepsilon + \mathcal{O}(\varepsilon^3).$ Here, $C_j^{(0)}:=\frac{\rho_c}{\overline{\kappa}_{b,j}}\operatorname{vol}(B_j)$ and we also set $q_{ij}=\varepsilon\frac{C_{j}^{(0)}}{4\pi d_{ij}} \asymp \varepsilon^{1-p}\frac{C_{j}^{(0)}}{4\pi \widetilde{d}}.$ Assuming that the bubbles have identical shapes and material properties, let \(C^{(0)} := C^{(0)}_j\). Thus, we denote $q^{(0)}_{ij} \asymp \frac{C_{j}^{(0)}}{4\pi \widetilde{d}},$ where the distance between any two bubbles scales as $d=\min\limits_{i\neq j}|z_i-z_j|= \widetilde{d}\ \varepsilon^p$ with $p>0$ and $i\ne j,$ with $\widetilde{d}$ is independent on $\varepsilon.$
\newline
From this, we finally obtain
\begin{align}\label{1}
    \|R_N\|_{H^r_{0,\sigma}} \leq \frac{\alpha_\infty^{N+1}}{1-\alpha_\infty \varepsilon^{1-p}}\varepsilon^{(N+1)(1-p)}\|V\|_{H^r_{0,\sigma}},
\end{align}
with
\begin{align}
    \alpha_\infty\leq \frac{\sigma_0^2}{\widetilde{d}\big(\omega_M^2\sigma_0^2-1\big)}(M-1)C^{(0)}.
\end{align}
We now focus on estimating the term $\|V\|_{H^r_{0,\sigma}}.$ We first recall its entry-wise frequency domain counter part i.e.
$$V_i(s) = \frac{s^2}{\omega_M^2s^2+1} F_i(s),$$
which can be written as
$$V_i(s) = R_{\omega_M^2}(s)F_i(s),\; \text{with}\ R_{\omega_M^2}(s) := \frac{s^2}{\omega_M^2s^2+1}\ \text{by}\ (\ref{defrs}).$$
Therefore, we see that for $F_i \in H_{r,\sigma}$, we derive that
\begin{align}
    \Vert V_i(s) \Vert_{H_{r,\sigma}}^2 
    = \frac{1}{2\pi}\int_{-\infty}^\infty |s|^{2r}\big|R_{\omega_M^2}(s)\big|^2\big\|F(\sigma+i\omega)\big\|_X^2dw
    \leq \big(\sup\limits_{\Re s= \sigma_0} \big|R_{\omega_M^2}(s)\big|^2\big) \Vert F_i(s) \Vert_{H_{r,\sigma}}^2,
\end{align}
which further implies due to the fact that \(\mathcal{L}\) is an isometry between \(H^r_{0,\sigma}\) and \(H_{r,\sigma},\) the following
\begin{align}\label{2}
    \Vert V_i(t) \Vert_{H^r_{0,\sigma}}^2 
    \le \big(\sup\limits_{\Re s= \sigma_0} \big|R_{\omega_M^2}(s)\big|^2\big) \Vert F_i(t) \Vert_{H^r_{0,\sigma}}^2 = \Big(\sup\limits_{\Re s= \sigma_0} \frac{|s|^2}{|\omega_M^2s^2+1|}\Big)^2 \Vert F_i(t) \Vert_{H^r_{0,\sigma}}^2.
\end{align}
Consequently, due to the estimate (\ref{2.10}), we have 
\begin{align}\label{3}
    \sup\limits_{\Re s= \sigma_0} |R_{\omega_M^2}(s)| \le \frac{\sigma_0^2}{\omega_M^2\sigma_0^2-1}.
\end{align}
Consequently, using (\ref{2}) and (\ref{3}) in (\ref{1}), we finally obtain that
\begin{align}
    \|R_N\|_{H^r_{0,\sigma}} \leq \frac{\alpha_\infty^{N+1}}{1-\alpha_\infty \varepsilon^{1-p}}\varepsilon^{(N+1)(1-p)}\frac{\sigma_0^2}{\omega_M^2\sigma_0^2-1}\|\mathbcal{F}\|_{H^r_{0,\sigma}},
\end{align}
with
\[\alpha_\infty\leq \frac{\sigma_0^2}{\widetilde{d}\big(\omega_M^2\sigma_0^2-1\big)}(M-1)C^{(0)}.\]
This completes the proof of the estimate (\ref{d}) in Theorem \ref{maintheorem}. \qed

\subsection{Truncation Error for the Scattered Field}\label{Tr} 
\noindent
We begin by recalling the asymptotic expansion for the scattered acoustic wave field $u^\textit{sc}$, as described in Proposition \ref{main-prop}:
\begin{align}\label{sc-expansion}
    u^{\textit{sc}}(x,t)
    =-\sum_{m=1}^M\frac{C_m}{4\pi|x-\mathbf{z}_m|}
    Y_m\!\big(t-c_0^{-1}|x-\mathbf{z}_m|\big)
    +\mathcal{O}(M\varepsilon^2),
\end{align}
which holds uniformly for $x\notin\mathcal{K}$.
\newline
Then, we substitute the Neumann-series representation for $\bm{Y}$ (\ref{Neumanntime}), \(Y_m(t)=\sum\limits_{n=0}^\infty (-1)^n(K^{*n}*V)_m(t)\), into the above asymptotic expansion to derive the following:
\begin{align}\label{sc-series}
    u^{\textit{sc}}(x,t)
    =-\sum_{m=1}^M\frac{C_m}{4\pi|x-\mathbf{z}_m|}
    \sum_{n=0}^\infty (-1)^n
    (K^{*n}*V)_m\!\big(t-c_0^{-1}|x-\mathbf{z}_m|\big)
    +\mathcal{O}(M\varepsilon^2).
\end{align}
Let us now define the truncated acoustic field $u^{\textit{sc},N}$, which we obtain by truncating the inner series at $n=N$. Then,
\begin{align}\label{sc-difference}
    u^{\textit{sc}}(x,t)-u^{\textit{sc},N}(x,t)
    =-\sum_{m=1}^M \frac{C_m}{4\pi|x-\mathbf{z}_m|}
    \sum_{n=N+1}^{\infty}(-1)^n
    (K^{*n}*V)_m\!\big(t-c_0^{-1}|x-\mathbf{z}_m|\big)
    +\mathcal{O}(M\varepsilon^2).
\end{align}
We define the remainder vector of the series for $\mathbf{Y}(t)$ as $\mathbf{R}_N(t):=\sum\limits_{n=N+1}^{\infty}(-1)^n K^{*n}*\mathbf{V}(t)$. In Section \ref{remainder}, we derive the following estimate for its $H_{0,\sigma}^r$-norm:
\begin{align}\label{R_N-bound}
    \|\mathbf{R}_N\|_{H^r_{0,\sigma}}
    \le
    \frac{\alpha_\infty^{N+1}}{1-\alpha_\infty\varepsilon^{1-p}}\, \varepsilon^{(N+1)(1-p)}
    \frac{\sigma_0^2}{\omega_M^2\sigma_0^2-1}\,\|\mathbcal{F}\|_{H^r_{0,\sigma}},
\end{align}
where we recall that
\[\alpha_\infty\leq \frac{\sigma_0^2}{\widetilde{d}\big(\omega_M^2\sigma_0^2-1\big)}(M-1)C^{(0)}.\]
Since each component satisfies $|Y_m(t)-Y_{m,N}(t)|\le \|\mathbf{R}_N\|_{\infty,[0,T]}$, substituting \eqref{R_N-bound} into \eqref{sc-difference} yields
\begin{align}\label{sc-diff-bound}
    \|u^{\textit{sc}}(x,t)-u^{\textit{sc},N}(x,t)\|_{H^r_{0,\sigma}}
    &\le
    \sum_{m=1}^M\frac{|C_m|}{4\pi|x-\mathbf{z}_m|}
    \|\mathbf{R}_N(\cdot-c_0^{-1}|x-z_m|)\|_{H^r_{0,\sigma}}
    +\mathcal{O}(M\varepsilon^2)\notag\\
    &\le
    \frac{M\, \mathbf{\mathrm{C}}^{(0)}\varepsilon}{4\pi\,d_{\operatorname{min}}}\, e^{-\sigma c_0^{-1}d_{\operatorname{min}}}
    \|\mathbf{R}_N\|_{H^r_{0,\sigma}}
    +\mathcal{O}(M\varepsilon^2),
\end{align}
where $d_{\operatorname{min}}:=\min\limits_m|x-\mathbf{z}_m|>0$ is fixed because $x \in \mathbb{R}^3\setminus \overline{D}$. Therefore, substituting the estimate \eqref{R_N-bound} into \eqref{sc-diff-bound} yields
\begin{align}
    \|u^{\textit{sc}}(x,t)-u^{\textit{sc},N}(x,t)\|_{H^r_{0,\sigma}}
    \nonumber&\le
    \frac{M\,C^{(0)}}{4\pi\,d_{\operatorname{min}}} \varepsilon^{(N+1)(1-p)+1} e^{-\sigma c_0^{-1}d_{\operatorname{min}}}
    \frac{\alpha_\infty^{N+1}}{1-\alpha_\infty\varepsilon^{1-p}}\,
    \frac{\sigma_0^2}{\omega_M^2\sigma_0^2-1}\,\|\mathbcal{F}\|_{H^r_{0,\sigma}}
   \\ & +\mathcal{O}(M\varepsilon^2).
\end{align}
This completes the proof of the estimate (\ref{e}) in Theorem \ref{maintheorem}. \qed

\section{Validating Foldy-Lax Approximation}     

\noindent
In this section, we aim to estimate the term $u^{sc,N}-u^{sc,N-1},$ which can be expressed as follows:
    \begin{align}\label{sc-difference}
         u^{\textit{sc},N}(x,t)-u^{\textit{sc},N-1}(x,t)
         =-\sum_{m=1}^M \frac{C_m}{4\pi|x-z_m|}(-1)^N (K^{*N}*V)_m\!\big(t-c_0^{-1}|x-z_m|\big).
    \end{align}
We note that this is the single $n=N$ term from the partial sums for each $\bm{Y}_m.$ 
\newline
We fix a path of indices $\gamma = (i_0, i_1, \ldots, i_N)$ of length $N$ and denote the product of the scattering coefficients along this path as $Q_\gamma := \prod\limits_{k=0}^{N-1} q_{i_k, i_{k+1}} $ and due to assumptions on the coefficient parameters, we have $Q_\gamma >0.$
\newline
Since we aim to understand the behavior of the scattered field at the $N^{\text{th}}$ interaction, it is natural to consider a specific path $\gamma$. This sequence of scatterings starts at $i_0$ via $V_{i_0}$ and ends at $i_N$, so the natural choice is to examine the $i_N$-th component of the vector $K^{*N}*V$ to isolate this path's contribution; any other component $m \neq i_N$ will not contain it. Consequently, we obtain that
    \begin{align}\label{sc1}
        \|u^{\textit{sc},N}(x,t)-u^{\textit{sc},N-1}(x,t)\|_{H^r_{0,\sigma}} 
        &\nonumber\ge \Big\Vert\frac{\mathbf{C}_{i_N}}{4\pi|x-z_{i_N}|} (K^{*N}*V)_{i_N}\!\big(\cdot-c_0^{-1}|x-z_{i_N}|\big)\Big\Vert_{H_{0,\sigma}^r}
         \\ & - \Vert\sum_{m\ne i_N} \frac{C_m}{4\pi|x-z_m|} (K^{*N}*V)_m\!\big(t-c_0^{-1}|x-z_m|\big)\|_{H^r_{0,\sigma}}.
    \end{align}

\noindent
Moreover, we recall the kernel entries $K_{ij}(t)$
    \begin{align}\label{kernelt1}
        K_{ij}(t)
        = q_{ij}\Big(\frac{1}{\omega_M^2}\delta(t-\tau_{ij})
        -\frac{1}{\omega_M^3}\sin\!\big(\tfrac{t-\tau_{ij}}{\omega_M}\big)\mathbf1_{t\ge\tau_{ij}}\Big),
    \end{align}
which consists of an atomic Dirac part and continuous part. Let us set $K^{(d)}_{ij}(t):=\frac{q_{ij}}{\omega_M^2}\delta(t-\tau_{ij})$ and $K^{(c)} := K_{ij}(t)-K^{(d)}_{ij}(t).$
We have the following equality
    \begin{align}
         \big(K^{(d)}\big)^{*N}_{ij}(t) 
         = \frac{1}{\omega_M^{2N}}\sum\limits_{\substack{\gamma=(i_0,i_1,\ldots,i_N) \\i_0=i,i_N=j}} \Big(\prod\limits_{k=0}^{N-1} q_{i_k,i_{k+1}} \Big) \delta \Big(t-\sum_{k=0}^{N-1}\tau_{i_k i_{k+1}}\Big).
    \end{align}
We skip the proof to Section \ref{appen}. Consequently, we obtain that
    \begin{align} \label{n11}
        (K^{*N}*V)_{i_N}(t) 
        = C_\gamma(t) + R_\gamma(t),
    \end{align}
where we denote
    \begin{align}
          C_\gamma(t) 
          := \frac{Q_\gamma}{\omega_M^{2N}} V_{i_0} (t-\tau_\gamma),\; \text{where}\ \tau_\gamma:=\sum_{k=0}^{N-1}\tau_{i_k i_{k+1}},
    \end{align}
and we split $R(t):=R_{\operatorname{atom}}+R_{\operatorname{cont}}$ with $R_{\operatorname{atom}}$ is the sum of contributions of all other atomic paths $\gamma\ne \gamma'$ and $R_{\operatorname{cont}}$ the sum of all contributions coming from words containing at least one $K^{(c)}$ factor.
\newline
We now take $H_{0,\sigma}^r-$norm to the both sides of (\ref{n11}) and use reverse triangle inequality to obtain
\begin{align}
    \Vert (K^{*N}*V)_{i_N} \Vert_{H_{0,\sigma}^r} \ge \Vert C_\gamma\Vert_{H_{0,\sigma}^r} - \Vert R_\gamma\Vert_{H_{0,\sigma}^r}.
\end{align}
\newline
We first consider the following term:
\begin{align}
    C_\gamma(t) := \frac{Q_\gamma}{\omega_M^{2N}} V_{i_0} (t-\tau_\gamma).
\end{align}
For each $s=\sigma+i\omega,$ taking Fourier-Laplace transform of both sides, we obtain
\begin{align}
    \widehat{C}_\gamma(s) = \frac{Q_\gamma}{\omega_M^{2N}} e^{-\tau_\gamma  s}\  \widehat{V}_{i_0}(  s).
\end{align}
We recall that $V_{i_0}(\cdot)$ has Laplace transform
\begin{align}
    \widehat{V}_{i_0}(  s) = \frac{  s^2}{\omega_M^2  s^2+1}\widehat{\mathcal{F}}(  s)\; \text{with}\ \Re   s\ge \sigma_0>\frac{1}{\omega_M}.
\end{align}
Here, we clearly see that, for the resonant angular frequency $\omega_{\operatorname{res}}:= \frac{1}{\omega_M}$, the denominator of the term $\widehat{V}_{i_0}(\cdot)$ vanishes on the imaginary axis $s=\pm i \omega_{\operatorname{res}}$.
\newline
We assume that $V_{i_0}(\cdot)$ has spectral energy in the resonance band $\Omega_{\operatorname{res}} := [\omega_{\operatorname{res}}-h,\omega_{\operatorname{res}}+h]$ in the sense that its Laplace transform satisfies
\begin{align}
|\widehat{V}_{i_0}(\sigma + i\omega)| \ge m,\quad \text{with constants}\ m, h>0.
\end{align}

\noindent
This means that if the incident wave has energy in a band containing the Minnaert frequency $\omega_{\operatorname{res}}$, then $\widehat{V}_{i_0}$ will also have enhanced spectral amplitude $m$ in $\Omega_{\operatorname{res}}$.
\newline
From the definition of the norm for the function space $H_{0,\sigma}^r$, we recall that
\begin{align}\label{imcond}
    \Vert f\Vert^2_{H_{0,\sigma}^r} \asymp \int_{-\infty}^\infty (\sigma^2+\omega^2)^r \big| f(\sigma+i\omega)\big|^2d\omega.
\end{align}
Consequently, we derive that
\begin{align}
    \Vert C_\gamma\Vert^2_{H_{0,\sigma}^r} 
    &\nonumber \asymp \int_{-\infty}^\infty (\sigma^2+\omega^2)^r \Big| \frac{Q_\gamma}{\omega_M^{2N}}e^{-\tau_\gamma(\sigma+i\omega)}\widehat{V}_{i_0}(\sigma+i\omega)\Big|^2d\omega
    \\ &\nonumber = \frac{Q_\gamma^2}{\omega_M^{4N}}e^{-2\sigma\tau_\gamma} \Vert \widehat{V}_{i_0}(\sigma+i\cdot)\Vert_{L^2_r(\mathbb{R})}
    \\ &\nonumber \gtrsim  \frac{Q_\gamma^2}{\omega_M^{4N}}e^{-2\sigma\tau_\gamma} \int_{\Omega_{\operatorname{res}}} (\sigma^2+\omega^2)^r m^2d\omega\qquad [\text{Due to condition}\ (\ref{imcond})] 
    \\ &\ge \frac{Q_\gamma^2}{\omega_M^{4N}}e^{-2\sigma\tau_\gamma} m^2 2h \min\limits_{\omega\in \Omega_{\operatorname{res}}}(\sigma^2+\omega^2)^r,
\end{align}
which further implies, after taking square roots
\begin{align}\label{vim}
    \Vert C_\gamma\Vert_{H_{0,\sigma}^r}  \ge \frac{Q_\gamma}{\omega_M^{2N}}e^{-\sigma\tau_\gamma} m \sqrt{2h} \Big(\min\limits_{\omega\in \Omega_{\operatorname{res}}}(\sigma^2+\omega^2)^r\Big)^\frac{1}{2}.
\end{align}
Subsequently, plugging the above estimate (\ref{vim}) in (\ref{sc1}), we finally obtain
\begin{align}
   \|u^{\textit{sc},N}(x,t)-u^{\textit{sc},N-1}(x,t)\|_{H^r_{0,\sigma}} 
    &\nonumber\ge \frac{C_{i_N}}{4\pi d_{\operatorname{max}}}\frac{Q_\gamma}{\omega_M^{2N}}e^{-\sigma(\tau_\gamma+c_0^{-1}d_{\operatorname{max}})} m \sqrt{2h} \ r_{\operatorname{min}}
    \\ \nonumber&- \frac{C_{i_N}}{4\pi d_{\operatorname{min}}} \|R(\cdot-c_0^{-1}|x-z_{i_N})\|_{H^r_{0,\sigma}}
    \\ &\nonumber - \Vert\sum_{m\ne i_N} \frac{C_m}{4\pi|x-z_m|}
    (K^{*N}*V)_m\!\big(t-c_0^{-1}|x-z_m|\big)\|_{H^r_{0,\sigma}}.
\end{align}
where we set $r_{\operatorname{min}}:= \big(\min\limits_{\omega\in \Omega_{\operatorname{res}}}(\sigma^2+\omega^2)^r\big)^\frac{1}{2}.$, and $d_{\operatorname{max}} := \max\limits_m |x - z_m| > 0$ (fixed, since \(x \notin \overline{D}\)) is the maximal observation distance.
\newline
Due to estimates derived in Section \ref{Tr}, we can show that
\begin{align}
    \frac{C_{i_N}}{4\pi d_{\operatorname{min}}}\|R_\gamma(\cdot-c_0^{-1}|x-z_{i_N})\|_{H^r_{0,\sigma}} \le \frac{C^{(0)}}{4\pi\,d_{\operatorname{min}}}\frac{\sigma_0^2}{\omega_M^2\sigma_0^2-1}\, \varepsilon^{N(1-p)+1}\ \alpha_\infty^{N+1}\, e^{-\sigma c_0^{-1}d_{\operatorname{min}}}
    \|\mathcal{F}\|_{H_{0,\sigma}^r}
\end{align}
and
\begin{align}
    &\nonumber \Vert\sum_{m\ne i_N} \frac{C_m}{4\pi|x-z_m|}
    (K^{*N}*V)_m\!\big(t-c_0^{-1}|x-z_m|\big)\|_{H^r_{0,\sigma}} 
    \\ &\le \frac{(M-1)\,C^{(0)}}{4\pi\,d_{\operatorname{min}}}\frac{\sigma_0^2}{\omega_M^2\sigma_0^2-1}\, \varepsilon^{N(1-p)+1}\ \alpha_\infty^{N+1}\, e^{-\sigma c_0^{-1}d_{\operatorname{min}}}\|\mathcal{F}\|_{H_{0,\sigma}^r}.
\end{align}
Subsequently, we derive that
\begin{align}
     \|u^{\textit{sc},N}(x,t)-u^{\textit{sc},N-1}(x,t)\|_{H^r_{0,\sigma}} 
    &\nonumber\ge \frac{C_{i_N}}{4\pi d_{\operatorname{max}}}\frac{Q_\gamma}{\omega_M^{2N}}e^{-\sigma(\tau_\gamma+c_0^{-1}d_{\operatorname{max}})} m \sqrt{2h}\ r_{\operatorname{min}}
    \\ &\nonumber - \frac{(M-1)\,C^{(0)}}{4\pi\,d_{\operatorname{min}}}\frac{\sigma_0^2}{\omega_M^2\sigma_0^2-1}\, \varepsilon^{N(1-p)+1}\ \alpha_\infty^{N+1}\, e^{-\sigma c_0^{-1}d_{\operatorname{min}}}\|\mathcal{F}\|_{H_{0,\sigma}^r}.
\end{align}
We now recall that $q_{ij} = \frac{C_j}{4\pi |z_i - z_j|}$ with $C_j = \frac{\rho_c}{\overline{\kappa}_{b_j}} \operatorname{vol}(B_j) \varepsilon.$ Since we have identical shapes and material properties of the bubbles, let us set \(C^{(0)} := C^{(0)}_j\). Corresponding to the path $\gamma$, we assume there exist $p\in (0,1)$ and constants $d_{i_k i_{k+1}}^{(0)}>0,$ independent on $\varepsilon,$ such that $|z_{i_k} - z_{i_{k+1}}| = d_{i_k i_{k+1}}^{(0)} \varepsilon^p$ for $k=0,1,\ldots,N-1.$
Thus, we obtain  
$$Q_\gamma := \prod\limits_{k=0}^{N-1} q_{i_k,i_{k+1}}  \asymp \varepsilon^{N(1-p)}Q^{(0)}_\gamma,\; \text{and}\; C_{i_N} \asymp \varepsilon C^{(0)},$$ where $Q^{(0)}_\gamma$ is a positive geometric constant independent on $\varepsilon.$ 
\begin{align}
     \|u^{\textit{sc},N}(x,t)-u^{\textit{sc},N-1}(x,t)\|_{H^r_{0,\sigma}} 
    &\nonumber\ge \Bigg( \frac{Q^{(0)}_\gamma C^{(0)}}{4\pi d_{\operatorname{max}}\omega_M^{2N}}\Bigg)\varepsilon^{N(1-p)+1}\ e^{-\sigma(\tau_\gamma+c_0^{-1}d_{\operatorname{max}})} m \sqrt{2h}\ r_{\operatorname{min}}
    \\ &\nonumber - \frac{(M-1)\,\mathbf{\mathrm{C}}^{(0)}}{4\pi\,d_{\operatorname{min}}}\frac{\sigma_0^2}{\omega_M^2\sigma_0^2-1}\, \varepsilon^{N(1-p)+1}\ \alpha_\infty^{N+1}\, e^{-\sigma c_0^{-1}d_{\operatorname{min}}}\|\mathcal{F}\|_{H_{0,\sigma}^r}.
\end{align}
Let us now denote the term $L_\gamma := \Bigg( \frac{Q^{(0)}_\gamma C^{(0)}}{4\pi d_{\operatorname{max}}\omega_M^{2N}}\Bigg)\varepsilon^{N(1-p)+1}\ e^{-\sigma(\tau_\gamma+c_0^{-1}d_{\operatorname{max}})} m \sqrt{2h}\ r_{\operatorname{min}}$ and $B:= \frac{(M-1)\,\mathbf{\mathrm{C}}^{(0)}}{4\pi\,d_{\operatorname{min}}}\frac{\sigma_0^2}{\omega_M^2\sigma_0^2-1}\, \varepsilon^{N(1-p)+1}\ \alpha_\infty^{N+1}\, e^{-\sigma c_0^{-1}d_{\operatorname{min}}}\|\mathcal{F}\|_{H_{0,\sigma}^r}.$
\newline
Thus, we obtain
\begin{align}
\|u^{\textit{sc},N}(x,t)-u^{\textit{sc},N-1}(x,t)\|_{H^r_{0,\sigma}}
\ge \Big(L_\gamma - (M-1)B\Big)\varepsilon^{N(1-p)+1}.
\end{align}
We now recall that the path $\gamma=(i_0,i_1,\ldots,i_N)$ is an ordered tuple of length $N$ such that $i_k\in {1,2,\ldots,M}.$ Let us now denote $\Gamma_N$ as the collection of all such ordered paths of length $N.$ It is clear that $\Gamma_N$ is a non-empty finite set with $\operatorname{card}(\Gamma_N)= M^{N+1}.$
\newline
Then, after observing the expression of the coefficient $L_\gamma$, selecting $C^{(0)}_{\operatorname{max}} = \max\limits_{j,\gamma}C^{(0)}_j$ and $Q^{(0)}_{\operatorname{max}} = \max\limits_{j,\gamma}Q^{(0)}_j,$ we see that the family $\{L_\gamma: \gamma\in \Gamma_N\}$ is bounded above.
\newline
This implies that the family $\{L_\gamma: \gamma\in \Gamma_N\}$ admits a maximum. Let us denote it by
\begin{align}
L^* = \max\limits_{\gamma\in \Gamma_N} L_\gamma
\end{align}
and therefore pick any maximizer $\gamma^*$ such that
\begin{align}
\gamma^* \in \arg \max\limits_{\gamma\in \Gamma_N} L_\gamma,
\end{align}
with $L_{\gamma^*}= L^*.$
\newline
Consequently, choosing any $\gamma^* \in \arg \max\limits_{\gamma\in \Gamma_N} L_\gamma$ provides us the best possible lower bound of the family $\{\varepsilon^{N(1-p)}\big(L_\gamma-(M-1)B\big): \gamma\in \Gamma_N\},$ which further implies
\begin{align}
\|u^{\textit{sc},N}(x,t)-u^{\textit{sc},N-1}(x,t)\|_{H^r_{0,\sigma}}
\ge \Big(L^* - (M-1)B\Big)\varepsilon^{N(1-p)+1}.
\end{align}
Next, we choose the parameters $m$ and $h$ corresponding to the resonance band in such a way that for $\theta \in (0,1)$
\begin{align}
   m\sqrt{2h}>(1-\theta)\big \Vert \mathbcal{F}(z_i,\cdot)\big\Vert_{H^r_{0,\sigma}} \dfrac{d_{\operatorname{max}}e^{\sigma\big(\tau_{\gamma^*}+c_0^{-1}(d_{\operatorname{max}}-d_{\operatorname{min}})\big)}}{d_{\operatorname{min}}Q^{(0)}_{\gamma^*} r_{\operatorname{min}}}\Big(\frac{C^{(0)}}{\widetilde{d}}\Big)^{N+1} \omega_M^{2N}\Big(\dfrac{(M-1)\sigma_0^2}{\omega_M^2\sigma_0^2-1}\Big)^{N+2},
\end{align}
which ensure 
\begin{align}
    L^* - (M-1)B>(1-\theta),
\end{align}
as well as the number of scatters satisfies the following bound
\begin{align}
M\leq M_{\max}=\Big\lfloor 1+\theta \frac{L^*}{B}\Big\rfloor.
\end{align}
We now recall the following estimate
\begin{align}\label{e}
    \|u^{\textit{sc}}(x,t)-u^{\textit{sc},N}(x,t)\|_{H^r_{0,\sigma}}
    &\nonumber \le
    \frac{M\,\mathbf{\mathrm{C}}^{(0)}}{4\pi\,d_{\operatorname{min}}} \varepsilon^{(N+1)(1-p)+1} e^{-\sigma c_0^{-1}d_{\operatorname{min}}}
    \frac{\alpha_\infty^{N+1}}{1-\alpha_\infty\varepsilon^{1-p}}\,
    \frac{\sigma_0^2}{\omega_M^2\sigma_0^2-1}\,\|\mathbcal{F}\|_{H^r_{0,\sigma}}
   \\ & +\mathcal{O}(M\varepsilon^2).
\end{align}
Then under the condition
\begin{align}
    p>1-\frac{1}{N},
\end{align}
we deduce the final estimate
\begin{align}
    \|u^{\textit{sc},N}(x,t)-u^{\textit{sc},N-1}(x,t)\|_{H^r_{0,\sigma}}
\ge\Big(L^* - (M-1)B\Big)\varepsilon^{N(1-p)+1} >> M\varepsilon^2.
\end{align}
This completes the proof. \qed

\section{Example: Air Bubbles in Water}
\noindent
We know that the natural frequency (Minnaert) of the air bubble created in the water is given by the following formula:
$$
f = \frac{1}{2\pi \varepsilon}\sqrt{\frac{3\gamma P_0}{\rho_c}}.
$$
Then $f\sim 3.26 \ \text{kHz}$. Here we consider $\gamma \sim 1.4$ (gas polytropic index), $P_0\sim 100\ \text{kPa}$ (ambient pressure), $\rho_c\sim 1000\ \text{kg/m}^3$ (density of the water), and $\varepsilon\sim 1\times 10^{-3}$ (radius of the bubble). Consequently, we consider the angular frequency as $\omega_{\operatorname{res}} = 2\pi f \sim 2.05\times 10^4\ \text{s}^{-1}$.
\newline
Then, we have $\omega_M = \frac{1}{\omega_{\operatorname{res}}} \sim 4.88\times 10^{-5}\ \text{s}^{-1}$. The speed of propagation outside the bubble is then given by $c_0 \sim 1480\ \text{m/s}$.
\newline
We now set $\sigma\sim 3\times 10^4\ \text{s}^{-1}$. We choose this to ensure that $\sigma>\frac{1}{\omega_M}$. Thus, we see
$$
\sigma^2\omega_M^2-1= 1.46^2-1= 2.14-1=1.14>0,
$$
which ensures the condition $\sigma^2\omega_M^2>1$. Therefore, we derive
$$
\frac{\sigma^2}{\sigma^2\omega_M^2-1} \sim 7.8\times 10^8\ \text{s}^{-2}.
$$
We now choose the lower bound of the Laplace magnitude of $V_{i_0}$ in the resonance band as $m=10$ (\ref{imcond}) and the bandwidth as $h=1\times 10^3\ \text{rad/s}$. Consequently, we obtain $m\sqrt{2h}\sim 444.721$.
\newline
We now fix the source at a distance $d_x \sim 0.1\ \text{m}$, i.e., very close to the cluster of bubbles.
\newline
We now fix $N=5$, $M=5$, and also consider $p=0.9$, which ensures the condition
$$
p>1-\frac{1}{N}=0.8.
$$
\newline
We now choose a path by $\gamma: 1\to 2\to 3\to 4\to 5\to 2$. Before proceeding, we know that for water $\rho_c=10^3\ \text{kg/m}^3$ and adiabatic air $\kappa_b\sim 1.4\times10^5\ \text{Pa}$. Now, due to the scaling property $\overline{\kappa}_{b,j} =\varepsilon^{-2}\kappa_{b,j}$, which further implies $\overline{\kappa}_{b,j} \sim 1.4\times 10^{11}\ \text{Pa}$. Now, if we consider the reference bubbles to be spherical with unit radius, then $\operatorname{vol}(B_j)=\frac{4}{3}\pi$. Thus, by recalling $C_j^{(0)} = \frac{\rho_c}{\overline{\kappa}_{b,j}}\operatorname{vol}(B_j) \sim 3\times 10^{-8}$.
\newline
We first consider that the centers of the bubbles are positioned as follows:
\begin{align*}
\text{Bubble centers (m):} \quad &
\mathbf{z}_1 = (0.00, 0.00, 0.00),\
\mathbf{z}_2 = (0.03, 0.00, 0.00),\
\mathbf{z}_3 = (0.03, 0.04, 0.00), \\
& \mathbf{z}_4 = (0.00, 0.05, 0.00),\
\mathbf{z}_5 = (-0.02, 0.02, 0.00).
\end{align*}
Now, as we consider the path $\gamma: 1\to 2\to 3\to 4\to 5\to 2$, we derive
\begin{align*}
\quad &
d_{12} = |\mathbf{z}_1 - \mathbf{z}_2| = 0.0300000000 \ \text{m} \\
& d_{23} = |\mathbf{z}_2 - \mathbf{z}_3| = 0.0400000000 \ \text{m} \\
& d_{34} = |\mathbf{z}_3 - \mathbf{z}_4| = \sqrt{0.03^2 + 0.01^2} = 0.0316227766 \ \text{m} \\
& d_{45} = |\mathbf{z}_4 - \mathbf{z}_5| = \sqrt{0.02^2 + 0.03^2} = 0.0360555128 \ \text{m} \\
& d_{52} = |\mathbf{z}_5 - \mathbf{z}_2| = \sqrt{0.05^2 + 0.02^2} = 0.0538516481 \ \text{m}.
\end{align*}
As we have the following asymptotic relation $|z_i-z_j| = d_{ij}^{(0)}\varepsilon^p$, we derive that
\begin{align*}
\quad & d_{12}^{(0)} = 14.0961839670,\ d_{23}^{(0)} = 18.7949119560,\ d_{34}^{(0)} = 14.8586825509, \\
& d_{45}^{(0)} = 16.9415046938,\ d_{52}^{(0)} = 25.3034246047,
\end{align*}
and correspondingly for $\tau_{ij} = c_0^{-1}|z_i-z_j|$ with $c_0=1480\ \text{m/s}$:
\begin{align*}
  \quad & \tau_{12} = 2.0270270270\times 10^{-5},\ \tau_{23} = 2.7027027027\times 10^{-5},\ \tau_{34} = 2.1366740947\times 10^{-5}, \\
& \tau_{45} = 2.4361832942\times 10^{-5},\ \tau_{52} = 3.6386248697\times 10^{-5};\quad \tau_{\gamma} = 1.2941211988\times 10^{-4}.
\end{align*}
We also recall that $q_{ij}=\varepsilon\frac{C_{j}^{(0)}}{4\pi d_{ij}} \approx \varepsilon^{1-p}\frac{C_{j}^{(0)}}{4\pi d^{(0)}_{ij}}$, where we set $q^{(0)}_{ij} \approx \frac{C_{j}^{(0)}}{4\pi d^{(0)}_{ij}}$. Therefore, we obtain that
\begin{align*}
\quad & q_{12}^{(0)} = 1.68907584246\times 10^{-10},\ q_{23}^{(0)} = 1.26680688184\times 10^{-10}, \\
& q_{34}^{(0)} = 1.60239804088\times 10^{-10},\ q_{45}^{(0)} = 1.40539605187\times 10^{-10},\ q_{52}^{(0)} = 9.40960529315\times 10^{-11},
\end{align*}
and consequently, $\tau_\gamma \sim 1.29\times 10^{-4}\ \text{s}$ and $Q_{\gamma}^{(0)} = 4.53419407817\times 10^{-50}$.
\newline
We now recall that (for simplicity we consider $r=0$ and $e^{-\sigma \tau_\gamma} \to 1$ as $\varepsilon\to 0$)
\begin{align}
    L_\gamma
    \nonumber&:= \Bigg( \frac{Q^{(0)}_\gamma C^{(0)}_{i_N}}{4\pi d_x\omega_M^{2N}}\Bigg)\ e^{-\sigma c_0^{-1}d_x} m \sqrt{2h}
    \\ &\nonumber= \frac{(4.53\times 10^{-50})\cdot (3\times 10^{-8})}{4\pi \cdot 0.1 \cdot (4.88\times 10^{-5})^{10}} e^{-\frac{(3\times 10^4)\cdot 0.1}{1480}}\cdot 444.72 \sim 8.28\times 10^{-13}.
\end{align}
We now recall
\begin{align}
    B
    \nonumber &:= \frac{C_j^{(0)}}{4\pi d_x}\frac{\sigma^2}{\omega_M^2\sigma^2-1}\ \alpha_\infty^{N+1}\ e^{-\sigma c_0^{-1}d_{\bm{x}}}\|\mathcal{F}\|_{H_{0,\sigma}^0}
    \\ &= \frac{3\times 10^{-8}}{4\pi\cdot 0.1}\cdot 7.8\times 10^8\cdot 0.13 \sim 0.32 \times \|\mathcal{F}\|_{H_{0,\sigma}^0}.
\end{align}
We now choose the overall source amplitude so that $$\|\mathcal{F}\|_{H_{0,\sigma}^0} \sim 5.59\times 10^{-13},$$
which further implies
$$M_{\operatorname{max}} = \big\lfloor 1+ 0.9 \cdot 4.6\big\rfloor =5.$$

\begin{remark}
\begin{enumerate}
    \item We remark that a "small" $\|\mathcal{F}\|_{H_{0,\sigma}^0}$-norm does not imply that the incident wave is vanishing; rather, it reflects a Laplace-weighted, second-derivative norm, not the amplitude.
\newline
Let us try to understand this fact with an example. Let us consider
    $$u^\textit{in} := A\sin(\omega_0 t)\bm{1}_{\ge 0}\implies F(t)=\partial_t^2u^\textit{in}(t)=-A\omega_0^2\sin(\omega_0 t)\bm{1}_{\ge 0}.$$
    Moreover, we can deduce the following:
    \begin{align}
      \|\mathcal{F}\|^2_{H_{0,\sigma}^0}
      \nonumber &= A^2\omega_0^4\int_{0}^\infty e^{-2\sigma t}\sin^2(\omega_0 t)\,dt
      \\ &= \frac{A^2\omega_0^4}{4}\Big(\frac{1}{\sigma}-\frac{\sigma}{\sigma^2+\omega_0^2}\Big).
    \end{align}
    Thus, we observe that for a fixed amplitude $A$, if $\sigma$ is large (which is the case here), $ \|\mathcal{F}\|_{H_{0,\sigma}^0} = \frac{A\omega_0^2}{2}\sqrt{\Big(\frac{1}{\sigma}-\frac{\sigma}{\sigma^2+\omega_0^2}\Big)}$ becomes small. Even if the wave is slowly varying, i.e., when $\omega_0\ll \sigma$, $\|\mathcal{F}\|_{H_{0,\sigma}^r} \sim \frac{A\omega_0^3}{2\sigma^{3/2}}$ also becomes small. 
    \item In addition, we recall that $R_{\omega_M^2}(s) = \frac{s^2}{\omega_M^2 s^2 + 1}$. It is easy to observe that at frequencies $\omega \sim \omega_{\operatorname{res}}$,
$$
|R_{\omega_M^2}(\sigma + i\omega)| \text{ is large, roughly } \sim \frac{\omega_{\operatorname{res}}^2}{2\sigma}.
$$
Next, we consider that $\widehat{F}$ has support only in the resonance band, and we define
$$
\widehat{F}(\sigma + i\omega) = \frac{m}{R_{\omega_M^2}(\sigma + i\omega)} \phi(\omega),
$$
where $\phi(\omega)$ is a smooth function with $0 \le \phi \le 1$, and $\phi \equiv 1$ on $[\omega_{\operatorname{res}} - h/2, \omega_{\operatorname{res}} + h/2]$, with support in $\Omega_{\operatorname{res}}$. Moreover,
$$
\widehat{V}(\sigma + i\omega) = R_{\omega_M^2}(\sigma + i\omega) \widehat{F}(\sigma + i\omega) = m \phi(\omega),
$$
which further implies
\begin{align}
    \|\mathcal{F}\|^2_{H_{0,\sigma}^0}
    &\nonumber= \int_{\Omega_{\operatorname{res}}} \frac{m^2}{|R_{\omega_M^2}(\sigma + i\omega)|^2} |\phi(\omega)|^2 \, d\omega
    \\ &\le \operatorname{vol}(\Omega_{\operatorname{res}}) \frac{m^2}{\inf\limits_{\omega \in \Omega_{\operatorname{res}}} |R_{\omega_M^2}(\sigma + i\omega)|^2} \sim m^2 2h \left( \frac{\sigma}{\omega_{\operatorname{res}}^3} \right)^2.
\end{align}
This shows that even though $\|\mathcal{F}\|_{H_{0,\sigma}^0}$ is ``small'', $\widehat{V}$ has a point-wise lower bound of $m$ across the resonance band.
\end{enumerate}
    
\end{remark}

\section{Appendix}\label{appen}
\begin{lemma}
Let $K$ be the matrix-valued distribution with entries  
\begin{align}\label{matrix K}
K_{ij}(t)
= q_{ij}\Big(\frac{1}{\omega_M^2}\delta(t-\tau_{ij})
-\frac{1}{\omega_M^3}\sin\!\big(\tfrac{t-\tau_{ij}}{\omega_M}\big)\mathbf{1}_{t\ge\tau_{ij}}\Big).
\end{align}
This consists of an atomic Dirac part and a continuous part. Let us set $K^{(d)}_{ij}(t):=\frac{q_{ij}}{\omega_M^2}\delta(t-\tau_{ij})$ and $K^{(c)}_{ij}(t) := K_{ij}(t)-K^{(d)}_{ij}(t).$
For $N\ge 1$ and $i,j=1,2,\ldots,M,$ we have the following equality in the distributional sense
    \begin{align}\label{kd}
         \big(K^{(d)}\big)^{*N}_{ij}(t) 
         = \frac{1}{\omega_M^{2N}}\sum\limits_{\substack{\gamma=(i_0,i_1,\ldots,i_N) \\i_0=i,i_N=j}} \Bigg(\prod\limits_{k=0}^{N-1} q_{i_k,i_{k+1}}\Bigg) \delta \Bigg(t-\sum_{k=0}^{N-1}\tau_{i_k i_{k+1}}\Bigg).
    \end{align} 
    Here, $\gamma$ denotes a path of length $N$ from $i$ to $j$.
\end{lemma}

\begin{proof}
    We prove the claim by induction on $N$. Before proceeding, we define the convolution of two $M\times M$ matrix-valued causal kernels $f$ and $g$ componentwise by
    \[
    \big(f*g\big)_{ij}(t)=\sum_{\ell=1}^M f_{i\ell}*g_{\ell j}(t)\; \text{in}\ \mathcal{D}'(0,\infty),
    \]
    where $f=[f_{ij}]$ and $g=[g_{ij}]$.
    \newline
    It is straightforward to observe that 
    \begin{align*}
        \big(K^{(d)}\big)^{*1}_{ij}(t) 
        &\nonumber=  K^{(d)}_{ij}
        \\ &= \frac{q_{ij}}{\omega_M^2}\delta(t-\tau_{ij}).
    \end{align*}
    This follows from the definition. Thus, the statement holds for $N=1$, i.e., for a path of length $1$ given by $\gamma=(i_0=i,i_1=j)$.
    \newline
    Now assume that the equality holds for $N$ in $\mathcal{D}'(0,\infty)$. Then,
    \begin{align}
        \big(K^{(d)}\big)^{*(N+1)}_{ij}(t) 
       \nonumber &= \sum_{\ell=1}^M\big(K^{(d)}\big)^{*N}_{i\ell}*K^{(d)}_{\ell j}(t)
       \\ \nonumber&= \sum_{\ell=1}^M\Bigg[\frac{1}{\omega_M^{2N}}\sum\limits_{\substack{\gamma=(i_0,i_1,\ldots,i_N) \\i_0=i,i_N=\ell}} \Bigg(\prod\limits_{k=0}^{N-1} q_{i_k,i_{k+1}}\Bigg) \delta \Big(t-\tau-\sum_{k=0}^{N-1}\tau_{i_k i_{k+1}}\Big)\Bigg]*\Big[\frac{q_{\ell j}}{\omega_M^2}\delta(\tau-\tau_{\ell j})\Big]
       \\ \nonumber&= \frac{1}{\omega_M^{2(N+1)}}\sum_{\ell=1}^M \sum\limits_{\substack{\gamma=(i_0,i_1,\ldots,i_N) \\i_0=i,i_N=\ell}} \Bigg(\prod\limits_{k=0}^{N-1} q_{i_k,i_{k+1}}\Bigg) q_{\ell j}\Bigg[ \delta \Big(t-\sum_{k=0}^{N-1}\tau_{i_k i_{k+1}}\Big)* \delta(t-\tau_{\ell j}\Big)\Bigg]
       \\ \nonumber&= \frac{1}{\omega_M^{2(N+1)}}\sum_{\ell=1}^M \sum\limits_{\substack{\gamma=(i_0,i_1,\ldots,i_N) \\i_0=i,i_N=\ell}} \Bigg(\prod\limits_{k=0}^{N-1} q_{i_k,i_{k+1}}\Bigg) q_{\ell j}\Bigg[ \delta \Big(t-\sum_{k=0}^{N-1}\tau_{i_k i_{k+1}}-\tau_{\ell j}\Big)\Bigg].
    \end{align}
    The path $\gamma = (i_0=i,i_1,\ldots,i_N=\ell)$ of length $N$, extended by $\ell\to j$, gives the path $\widetilde{\gamma} := (i_0=i,i_1,\ldots,i_N=\ell, i_{N+1}=j)$ of length $N+1$. Thus, we obtain
    \begin{align}
        \big(K^{(d)}\big)^{*(N+1)}_{ij}(t)  = \frac{1}{\omega_M^{2(N+1)}}\sum\limits_{\substack{\widetilde{\gamma}=(i_0,i_1,\ldots,i_{N+1}) \\i_0=i,i_{N+1}=j}} \Bigg(\prod\limits_{k=0}^{N} q_{i_k,i_{k+1}}\Bigg) \delta \Bigg(t-\sum_{k=0}^{N}\tau_{i_k i_{k+1}}\Bigg).
    \end{align}
    This completes the proof by induction.
\end{proof}

\end{document}